\documentclass[11pt]{article}
\usepackage{graphicx}
\usepackage{amsfonts}
\usepackage{amsmath}
\usepackage{amssymb}
\usepackage{amsthm}
\usepackage{enumitem}
\usepackage[top=1in,bottom=1in,left=1in,right=1in]{geometry}
\usepackage{bbm}
\usepackage{booktabs}
\usepackage{hyperref}

\DeclareUnicodeCharacter{2032}{\ensuremath{'}}

\theoremstyle{plain}
\newtheorem{theorem}{Theorem}[section]
\newtheorem{proposition}[theorem]{Proposition}
\newtheorem{lemma}[theorem]{Lemma}
\newtheorem{corollary}[theorem]{Corollary}

\theoremstyle{definition}
\newtheorem{example}[theorem]{Example}

\theoremstyle{remark}
\newtheorem{remark}[theorem]{Remark}

\usepackage[giveninits=true,backend=biber, maxnames=4,maxalphanames=4,style=alphabetic,date=year,doi=true,isbn=true,url=false,eprint=false]{biblatex}
\DeclareSourcemap{
  \maps[datatype=bibtex]{
    \map[overwrite=true]{
      \step[fieldsource=shortjournal, final]
      \step[fieldset=journaltitle, origfieldval]
    }
    \map{
      \step[fieldsource=doi, final]
      \step[fieldset=eprint, null]
      \step[fieldset=eprinttype, null]
      \step[fieldset=eprintclass, null]
      \step[fieldset=url, null]
      \step[fieldset=urldate, null]
    }
    \map{
      \step[fieldsource=eprint, final]
      \step[fieldset=url, null]
      \step[fieldset=urldate, null]
      \step[fieldsource=eprint, match=\regexp{\s*\[[^\]]*\]}, replace={}]
    }
    \map{
      \step[fieldsource=note, match=\regexp{_eprint}, final]
      \step[fieldset=note, null]
    }
    \map{
      \step[fieldset=file, null]
      \step[fieldset=abstract, null]
    }
    \map{
      \step[fieldsource=note, match=\regexp{\Atex\.note:\s*}, replace={}]
    }
  }
}
\renewbibmacro{in:}{}
\AtEveryBibitem{%
  \clearlist{language}%
  \clearfield{issn}
}

\newcommand{\SampleFirst}{1925-12}
\newcommand{\SampleLast}{2024-12}
\newcommand{\SampleFirstYear}{1925}
\newcommand{\SampleLastYear}{2024}
\newcommand{\NFirms}{26{,}100}
\newcommand{\NFirmMonths}{3{,}751{,}795}
\newcommand{\AvgFirms}{3{,}155}
\newcommand{\NEnt}{19{,}567}
\newcommand{\NExit}{16{,}253}
\newcommand{\NEntNew}{18{,}747}
\newcommand{\NEntSpin}{820}
\newcommand{\NExitMerger}{6{,}796}
\newcommand{\NExitDropped}{8{,}879}
\newcommand{\NExitLiq}{578}
\newcommand{\PctExitMerger}{42}
\newcommand{\NTechDropped}{8{,}577}
\newcommand{\PctExitRawTechnical}{27}
\newcommand{\PctExitBottomFive}{31}
\newcommand{\PctExitBottomFiveOrg}{53}
\newcommand{\MedRankMerger}{0.53}
\newcommand{\MedRankDrop}{0.04}
\newcommand{\NExitLiqWin}{477}
\newcommand{\RateBirthMonthly}{5.78}
\newcommand{\RateDeathMonthly}{4.98}
\newcommand{\RateDeathAnnualPct}{6.0}
\newcommand{\NetGrowthAnnualPct}{1.0}
\newcommand{\PoisChiDfEnt}{1.1}
\newcommand{\PoisChiDfExit}{1.4}
\newcommand{\NBAlphaEnt}{0.00}
\newcommand{\NBAlphaExit}{0.00}
\newcommand{\NBKSEnt}{0.004}
\newcommand{\NBKSExit}{0.012}
\newcommand{\FCross}{0.63}
\newcommand{\DfSplineEnt}{8}
\newcommand{\DfSplineExit}{8}
\newcommand{\NBins}{100}

\newcommand{\MuTildeMonthly}{0.0049}
\newcommand{\MuTildeAnnual}{0.059}
\newcommand{\SigmaMonthly}{0.135}
\newcommand{\SigmaAnnual}{0.467}
\newcommand{\InitMonth}{1975-01}
\newcommand{\NMain}{10{,}000}
\newcommand{\SupErrSixty}{0.010}
\newcommand{\SupErrOneTwenty}{0.013}
\newcommand{\SupErrTwoForty}{0.011}
\newcommand{\SweepErrSmall}{0.052}
\newcommand{\SweepErrBig}{0.010}
\newcommand{\NSweepSmall}{500}
\newcommand{\NSweepBig}{20{,}000}
\newcommand{\GrowthSim}{1.185}
\newcommand{\GrowthTheory}{1.211}
\newcommand{\WaveSpeedAnnual}{0.100}
\newcommand{\TurnoverCorrAnnual}{+0.0412}

\newcommand{\DecBirthRateMin}{1.8}
\newcommand{\DecBirthRateMax}{9.5}
\newcommand{\DecDeathRateMin}{0.3}
\newcommand{\DecDeathRateMax}{7.4}
\newcommand{\TwoThousandsBirthRate}{4.1}
\newcommand{\TwoThousandsDeathRate}{7.4}

\newcommand{\CEmpAnnual}{0.085}
\newcommand{\MeasVolBottom}{0.98}
\newcommand{\MeasVolTop}{0.31}

\newcommand{\GrowthTotalFull}{8.3}
\newcommand{\GrowthMedianFull}{4.3}
\newcommand{\GrowthCountFull}{2.6}
\newcommand{\GrowthShapeFull}{1.4}
\newcommand{\GrowthTotalRecent}{10.0}
\newcommand{\GrowthMedianRecent}{-1.2}
\newcommand{\GrowthShapeRecent}{9.5}

\newcommand{\ZCILHundred}{0.25}
\newcommand{\CSILHundred}{0.89}
\newcommand{\BMILHundred}{0.19}
\newcommand{\RECILHundred}{0.11}
\newcommand{\RECILThousand}{0.09}
\newcommand{\ZCConv}{0.0009}
\newcommand{\CSConv}{0.0008}
\newcommand{\WidthData}{3.8}
\newcommand{\NDecAll}{10}
\newcommand{\DecErrMin}{0.15}
\newcommand{\DecErrMed}{0.29}
\newcommand{\DecErrMax}{0.60}
\newcommand{\DecErrRecMin}{0.05}
\newcommand{\DecErrRecMed}{0.22}
\newcommand{\AlphaData}{1.06}
\newcommand{\AlphaRecWave}{1.11}
\newcommand{\AlphaZCWave}{1.53}
\newcommand{\AlphaDataHill}{1.07}
\newcommand{\AlphaRecWaveHill}{1.13}

\newcommand{\DriftRelBottom}{-0.26}
\newcommand{\DriftRelMid}{-0.06}
\newcommand{\DriftRelTop}{-0.04}
\newcommand{\DriftChordDevMax}{0.031}
\newcommand{\RelaxHalfBottomMonths}{6}
\newcommand{\RelaxHalfTopYears}{53}
\newcommand{\EOneFailShareMeas}{89}

\newcommand{\RelaxEmpBottomYears}{10}
\newcommand{\RelaxEmpTopYears}{37}
\newcommand{\RelaxEmpBottomDemeanMonths}{7}
\newcommand{\RelaxEmpTopDemeanMonths}{15}

\newcommand{\PortEWExcess}{17.5}
\newcommand{\PortEWTurn}{-8.1}
\newcommand{\PortEWDrag}{6.0}
\newcommand{\PortEWRel}{3.4}
\newcommand{\PortMktTurn}{1.0}
\newcommand{\PortMktDrag}{1.5}
\newcommand{\PortMktRel}{-0.5}
\newcommand{\PortEWMinusMkt}{3.9}
\newcommand{\PortMktDilution}{1.5}
\newcommand{\PortEntryCountPct}{6.9}
\newcommand{\PortTopRankRel}{1.0}
\newcommand{\PortTopCapExc}{6.2}
\newcommand{\PortTopCapLeak}{5.2}
\newcommand{\PortTopCapRel}{1.0}
\newcommand{\PortTopCapCountStart}{679}
\newcommand{\PortTopCapCountPeak}{907}
\newcommand{\PortTopCapPeakYear}{1996}
\newcommand{\PortTopCapCountLate}{407}
\newcommand{\PortRecEWCents}{88}
\newcommand{\PortRecDropCents}{76}
\newcommand{\PortAcctResidMax}{0.2}

\newcommand{\BootReps}{1000}
\newcommand{\BootBandPct}{90}
\newcommand{\BootQStarLo}{0.61}
\newcommand{\BootQStarHi}{0.65}
\newcommand{\BootHalfBotLoMonths}{5.5}
\newcommand{\BootHalfBotHiMonths}{6.4}
\newcommand{\BootHalfTopLoYears}{46}
\newcommand{\BootHalfTopHiYears}{61}
\newcommand{\BootAlphaLo}{1.08}
\newcommand{\BootAlphaHi}{1.14}
\newcommand{\CenturyQStar}{0.61}
\newcommand{\CenturyHalfTopYears}{53}
\newcommand{\CenturyHalfBotMonths}{7}
\newcommand{\RecentFifteenFPrimeTopAnnual}{-0.56}
\newcommand{\RecentFifteenAlpha}{0.98}
\newcommand{\RecentFifteenHalfTopYears}{124}

\newcommand{\R}{\mathbb{R}}
\renewcommand{\P}{\mathbb{P}}
\newcommand{\E}{\mathbb{E}}
\renewcommand{\d}{\mathrm{d}}
\newcommand{\lb}{\left}
\newcommand{\rb}{\right}
\newcommand{\ind}{\mathbbm{1}}
\newcommand{\pp}{\partial}

\makeatletter
\def\@maketitle{%
  \newpage
  \begin{center}%
  \let \footnote \thanks
    {\LARGE \@title \par}%
    \vskip 1.5em%
    {\large
      \lineskip .5em%
      \begin{tabular}[t]{c}%
        \@author
      \end{tabular}\par}%
    \vskip 1em%
    {\large \@date}%
  \end{center}%
  \par
  \vskip 1.5em}
\makeatother

\title{Traveling Waves in Equity Markets with\\Rank-Based Entry and Exit}
\author{
Graeme Baker\footnote{Department of Statistics, Columbia University, NY, USA \href{mailto:g.baker@columbia.edu}{g.baker@columbia.edu} (Corresponding Author)}\and Caroline Smyth\footnote{Department of Mathematics, Columbia University (Class of 2026), NY, USA \href{mailto:css2215@columbia.edu}{css2215@columbia.edu}} }
\date{}

\begin{document}
\maketitle

\begin{abstract}
We model equity markets using geometric Brownian particles entering and exiting at rank-dependent intensities. In the many-firm limit, the capital distribution converges to the solution of a reaction-diffusion equation with reaction term built from the intensities. Calibrated on CRSP data, the reaction term is bistable, and the long-run distribution is a traveling wave: we prove existence, uniqueness, and, for constant coefficients, exponential relaxation. Turnover, not drift, stabilizes the calibrated market. With measured volatility, the wave tracks the empirical capital distribution in every decade, determines the capitalization growth of diversity-weighted portfolios, and places the market just inside the boundary of the diverse phase. Turnover reclaims most of what rebalancing gains.
\end{abstract}

\begin{quotation}\small
\end{quotation}

\section{Introduction}

Equity markets show a remarkable regularity: the distribution of market capitalizations, from the largest firm down to the smallest, keeps roughly the same shape over decades despite regular firm turnover. Rank-based models of capital distributions, such as the Atlas model developed in \cite{fernholz_stochastic_2002,banner_atlas_2005}, explain this stability through rank-dependent growth rates and volatilities: small firms grow faster on average, which keeps the distribution from spreading out. These models, however, hold the set of firms fixed. In reality, firms enter the market through listings and spin-offs and leave it through failures and mergers, and this turnover is substantial: in the data we study below, the average firm faces roughly a \RateDeathAnnualPct\% chance of exiting in a given year.

Herein, we add entry and exit to the rank-based diffusion framework by a simple mechanism. We model the market as a system of particles, where each particle follows an independent geometric Brownian motion, and additionally, new particles are created and existing particles removed at rates that depend on the particle's rank within the current population. Our baseline model is deliberately parsimonious: apart from the constant drift $\mu$ and volatility $\sigma$, it is specified by two functions of rank, a birth intensity $\lambda_b$ and a death intensity $\lambda_d$, and both can be estimated directly from event data. 

Three findings organize the paper. First, calibrated on a half-century of Center for Research in Security Prices (CRSP) event data, the reaction term produced by entry and exit is bistable: deaths outpace births among the smallest firms, and entries outpace exits through the upper-middle ranks. Second, the bistable equation admits a unique traveling wave, which is the observed stable shape of the capital distribution. Relaxation toward the wave is bounded by the endpoint slopes of the reaction term, and the data read the slopes as two timescales: months at the bottom of the market, decades at the top. Third, the drift curve recovered from the data violates the chord condition of \cite{jourdain_capital_2015} at every rank, so for the measured coefficients a drift-stabilized equilibrium does not exist at all (Subsection~\ref{subsec:beyond}). Turnover is the existence mechanism for the wave, not a correction: removing it from the calibrated model removes convergence altogether. We say that the market is \emph{turnover-stabilized}, in contrast with previous literature on volatility-stabilized \cite{fernholz_relative_2005} and drift-stabilized markets \cite{banner_atlas_2005, jourdain_capital_2015}.

Our story proceeds in three steps. The mean-field limit of the particle system describes the market by a tractable partial differential equation (PDE). The PDE is calibrated directly from observable data: the intensities are fitted from event counts, the volatility is measured rank by rank, and the drift is recovered from a traveling-wave identity. We then use the calibrated model in two ways: we test it against a century of market growth, and we evaluate the capitalization growth of common passive portfolio rules along its wave, measuring the competition between rebalancing and turnover. Our contributions below follow this arc:

\begin{enumerate}[itemsep=1pt]
    \item We prove a limit theorem (Theorem~\ref{thm:mean-field}): as the number of firms grows, the empirical distribution of capitalizations converges to the solution of the semilinear PDE
    \begin{equation*}
        \pp_t u \;=\; \frac{\sigma^2 x^2}{2}\,\pp^2_{xx} u + (\sigma^2 - \mu)\, x \,\pp_x u + \tilde f(u),
    \end{equation*}
    where the reaction term $\tilde f$ is an explicit integral of $\lambda_b - \lambda_d$. The theorem also allows the drift and volatility to depend on rank, making the operator quasilinear as in the rank-based models of \cite{shkolnikov_large_2012, jourdain_propagation_2013}. The convergence is in the Kolmogorov distance (the uniform distance between distribution functions), uniformly on compact time intervals, where the prior limits without turnover are in the weak topology. For the long-run theory, Theorem~\ref{thm:wave} gives existence and uniqueness of the wave's speed and profile at the same generality, Corollary~\ref{cor:tail} identifies the wave's Pareto tail exponent, and Theorem~\ref{thm:relaxation} shows exponential sup-norm convergence to the wave for constant coefficients, at a rate bounded by the reaction term's endpoint slopes alone (Remark~\ref{rem:rank-dependent} states our expectation for the rank-dependent case).
        
   \item We calibrate $\lambda_b$ and $\lambda_d$ on the post-expansion half-century of CRSP monthly data, with exits identified from CRSP's recorded delisting reasons (Section~\ref{sec:empirics}).
The calibrated reaction term is bistable across decades, with the exception of the 2000s (Subsection~\ref{subsec:beyond} records which conclusions survive in that case). The volatility curve is measured rank by rank from returns; the drift, which time series cannot usefully identify (see Remark~\ref{rem:drift-estimation}), is recovered from the wave identity and yields a tight fit of the capital distribution curve, the log-log plot of market weights against rank (Subsection~\ref{subsec:drift-calib}).

    \item The shape of the wave, computed with the measured volatility curve and a constant, uncalibrated drift, tracks the empirical capital distribution curve, decade by decade as well as over the century (Subsection~\ref{subsec:ranked-vol}); no parameter is fitted to the curve itself.  An exact growth accounting shows that, outside of concentration episodes, the market grows by translation of the wave plus the entry of firms (Subsection~\ref{subsec:growth-acct}). The relaxation timescales agree with the century of data: persistence of shape deviations rises from months at the bottom of the market to decades in the top percent, in the order and on the scale predicted.

    \item The traveling wave explains the relative capitalization growth of passive portfolios. Along the wave, the growth rate of any $p$-diversity-weighted portfolio below a critical diversity index is an explicit integral against the profile (Proposition~\ref{prop:portfolio}). The wave identity eliminates the drift, the reduction formula of \cite{jourdain_capital_2015} gains two turnover corrections, and that paper's critical diversity index becomes the tail exponent of the wave. The index is an estimated functional of the model rather than an assumption: the calibrated market is diverse by a thin margin, and the margin vanishes on the coefficients of the most recent fifteen years (Remark~\ref{rem:diversity-loss}). With the calibrated coefficients, the equally weighted portfolio gains the excess growth rate of the variance curve, but turnover takes most of it back, while the market portfolio grows near the speed of the front. Decade by decade, the integral identity at $p = 1$ closes the market portfolio's accounting to within \PortAcctResidMax\% per year, with the concentration episodes carried by the shape term of the growth identity.
\end{enumerate}

\subsection{Background}
\label{subsec:background}

We briefly recall the basic units of our model and their precedents in the literature.

\paragraph{Geometric Brownian Motion.}
The classical model for a single stock price or market capitalization $X_t$ is the geometric Brownian motion
\begin{equation}\label{eq:gbm}
    \d X_t \;=\; \mu X_t \,\d t + \sigma X_t \,\d W_t,
\end{equation}
where $W$ is a standard Brownian motion, $\mu \in \R$ is the growth rate, and $\sigma > 0$ is the volatility. We recall this basic example to stress that the two parameters are on very different statistical footing: the volatility can be estimated from the quadratic variation of past data (realized volatility), with precision determined by the sampling frequency rather than by the length of the window. In contrast, the growth rate can only be learned from the length of the window, and even a century of data determines it loosely (see, for instance, \cite[Section 12.6]{bjork_arbitrage_2020}). This asymmetry persists to rank-dependent settings and will guide the modeling choices of Section~\ref{sec:wave} below.

By It\^o's formula, $Y_t = \log X_t$ is a Brownian motion with constant drift:
\begin{equation*}
    \d Y_t \;=\; \tilde\mu \,\d t + \sigma \,\d W_t, \qquad \tilde\mu = \mu - \frac{\sigma^2}{2}.
\end{equation*}
If many firms evolve independently by \eqref{eq:gbm}, the empirical fraction of firms with capitalization at most $x$ converges, by the law of large numbers, to a deterministic limit $u(t,x)$. Since $u$ is the cumulative distribution function (CDF) of the log-normal position of a single firm, it solves a linear parabolic equation, namely \eqref{eq:pde} below with $\tilde f = 0$.

\paragraph{Rank-Based Models.}
The Atlas model \cite{banner_atlas_2005} of stochastic portfolio theory (SPT) has the drift and volatility in \eqref{eq:gbm} depend on the firm's rank: the smallest firm receives an extra push, which stabilizes the shape of the capital distribution.
For background reading on Atlas models and SPT, we direct the reader towards \cite{fernholz_stochastic_2002, karatzas_stochastic_2009}. \cite{shkolnikov_large_2012} and \cite{jourdain_propagation_2013} show that as the number of firms tends to infinity, the empirical CDF in such models converges to the solution of a porous medium equation; the mean-field Atlas description is developed further through free boundary problems in \cite{atar_free_2025} and through SDEs with boundary interaction in \cite{jettkant_atlas_2025}. Our limit theorem allows for rank-dependent drift and volatility, and also extends the rank-dependence to entry and exit rates.
Reconciling rank-based models with empirical data continues to be of interest: \cite{itkin_calibrated_2026} calibrate rank-based volatility-stabilized models to CRSP capitalizations, and \cite{campbell_macroscopic_2025} document the macroscopic stylized facts that such calibrations target.
The calibrations mentioned above choose parameters to reproduce the curve, whereas the curve is a prediction of our model. In Subsection~\ref{subsec:ranked-vol}, we show that the long-run capital distribution curve is produced by turnover using coefficients measured directly from the data, and with no parameter fitted to the curve itself. 

\paragraph{Long-Run Behaviour.}
In the finite Atlas model, the gaps between the ranked log capitalizations have a stationary distribution and the configuration itself drifts upward at a common rate: the market is a traveling wave in log scale \cite{banner_atlas_2005}, with convergence rates for the finite system studied in \cite{ichiba_convergence_2013}. The stationarity rests on the theory of semimartingale reflected Brownian motion and inherits a delicate balance among the coefficients (see the survey \cite{williams_semimartingale_1995} and the references within; recent joint work of the first author addresses finite-time breakdown, stationarity, and self-similar solutions in the mean-field case with a critical parameter \cite{baker_particle_2026}). In the mean-field limit, \cite{jourdain_propagation_2013, jourdain_capital_2015} establish traveling-wave behaviour and explicit stationary capital distributions for the Atlas model, recovering a phase transition in the shape of the capital distribution earlier identified by \cite{chatterjee_phase_2010} (Section~\ref{sec:portfolio} identifies the critical index of that transition with the tail exponent of our wave). In all of these models, stabilization comes from the drift and volatility structure. In our model with bistable turnover, the system admits a unique traveling wave with no chord condition. The complete comparison is in Remark~\ref{rem:atlas} and Subsection~\ref{subsec:beyond}. From the viewpoint of reaction-diffusion theory this is the expected behaviour: the zeroth-order term determines the states connected by a front and, through its endpoint slopes, the spectral gap behind exponential relaxation, while the drift and the diffusion shape the profile between them; without a reaction term, the equation is a conservation law, stable by contraction, with no spectral gap and no rate (Subsection~\ref{subsec:beyond}).

\paragraph{Entry and Exit.}
Entry and exit have begun to receive attention in stochastic portfolio theory. On the market's side, \cite{karatzas_diverse_2016} build composition change into the dynamics, with competing Brownian particles that split and merge, keeping a finite, fluctuating number of firms; their splits and mergers are designed to sustain market diversity, while our entry and exit rates are estimated from historical data. On the investor's side, the open-markets literature \cite{karatzas_open_2021, itkin_open_2024} confines trading to the top $N$ stocks, whose membership changes over time. Our Sections~\ref{sec:empirics} through \ref{sec:wave} build a market model from the event data, and Section~\ref{sec:diversity} computes the growth of common portfolios using the model, with the leakage of top-set rules appearing as a boundary term (Remark~\ref{rem:leakage}).

\paragraph{Branching Particle Systems.}
Particle systems with births and deaths are classical objects. \cite{chauvin_stochastic_1990} use branching Brownian motions to represent and simulate solutions of reaction-diffusion equations. The connection between branching Brownian motions and the Fisher--KPP equation with its traveling-wave solutions has been well studied (see \cite{mckean_application_1975} and the monograph \cite{bramson_convergence_1983}). Our reaction term is bistable rather than monostable; the stability theory for bistable waves is classical \cite{fife_approach_1977}, and existence for the quasilinear class is covered in \cite{gilding_travelling_2004}. In the classical representations, the branching rate is a function of the particle's own position and the particles evolve independently, so the reaction term of the limiting equation is fixed in advance. For our setting, the rate depends on the particle's rank, that is, on the empirical distribution of the whole system, and the reaction term is therefore evaluated along the limiting distribution itself.

\section{A Particle System with Rank-Based Births and Deaths}

\subsection{The Model}
\label{subsec:finite-system}

Throughout, we fix a filtered probability space $(\Omega, \mathcal F, (\mathcal F_t)_{t \geq 0}, \P)$, and we work under the historical measure. 
Fix two continuous functions $\lambda_b, \lambda_d : [0,1] \to [0,\infty)$, the birth and death intensities, and coefficient functions $b : [0,1] \to \R$ and $\sigma : [0,1] \to (0,\infty)$. The system starts from $n$ particles with positions $X^1_0, \ldots, X^n_0 \in (0,\infty)$, sampled independently from a distribution with continuous CDF $u_0$. Each particle represents the market capitalization of one firm. Writing $n_t$ for the number of particles alive at time $t$, the dynamics are:
\begin{enumerate}[itemsep=1pt]
    \item \textbf{Diffusion.} Between events, each particle diffuses with rank-dependent coefficients,
    \begin{equation*}
        \d X^i_t \;=\; b(q_i(t))\, X^i_t \,\d t + \sigma(q_i(t))\, X^i_t \,\d W^i_t,
    \end{equation*}
    driven by its own independent Brownian motion. Between events the coefficients are piecewise constant in log scale (they change only when particles change rank), so the system of diffusions is well posed by the results of Bass and Pardoux \cite{bass_uniqueness_1987}; see also the discussion in \cite{banner_atlas_2005}.
    \item \textbf{Rank.} The rank of particle $i$ at time $t$ is its normalized position in the current population,
    \begin{equation*}
        q_i(t) \;=\; \frac{1}{n_t} \,\#\{j \text{ alive} : X^j_t \leq X^i_t\} \in (0, 1],
    \end{equation*}
    so $q_i \approx 0$ for the smallest firms and $q_i = 1$ for the largest.
    \item \textbf{Births and Deaths.} Particle $i$ carries two independent exponential clocks. At rate $\lambda_b(q_i(t))$ a new particle is created at the same position $X^i_t$; at rate $\lambda_d(q_i(t))$ the particle is removed.
\end{enumerate}
Placing the newborn at the parent's position is a modeling convenience; Remark~\ref{rem:immigration} explains why the large-population limit does not depend on it. Since the total event rate per particle is at most $\Lambda := \sup \lambda_b + \sup \lambda_d < \infty$, the population is dominated by a birth process with rate $\Lambda$ and cannot explode in finite time.

The macroscopic objects we study are the scaled empirical measure, its scaled CDF, and total mass:
\begin{equation*}
    \mu^n_t \;=\; \frac{1}{n} \sum_{i \text{ alive}} \delta_{X^i_t},
    \qquad
    F^n(t,x) \;=\; \mu^n_t\big((0, x]\big),
    \qquad
    m^n_t \;=\; F^n(t, \infty) \;=\; \frac{n_t}{n}.
\end{equation*}
Note that $F^n$ is normalized by the \emph{initial} population size $n$, so its total mass $m^n_t$ moves as firms enter and exit, while the ranks $q_i$ are computed within the \emph{current} population. Keeping these two normalizations distinct is what allows the limit to take a clean form.

\paragraph{The Reaction Term.}
Before stating the theorem, we explain where the nonlinearity comes from. Almost surely, at Lebesgue-almost every instant the ranks of the $n_t$ living particles are exactly the grid $\{1/n_t, 2/n_t, \ldots, 1\}$: ranks are always uniformly spread, no matter what the positions are. The total rate of death events at positions $\leq x$ is therefore
\begin{equation*}
    \sum_{i:\, X^i_t \leq x} \lambda_d(q_i(t))
    \;=\; n_t \cdot \frac{1}{n_t}\sum_{k=1}^{\lfloor n_t F^n(t,x)/m^n_t \rfloor} \lambda_d\lb(\tfrac{k}{n_t}\rb)
    \;\approx\; n \, m^n_t \int_0^{F^n(t,x)/m^n_t} \lambda_d(r)\, \d r,
\end{equation*}
a Riemann sum of the intensity over the ranks below $x$, and similarly for births. Defining
\begin{equation*}
    f(q) \;=\; \int_0^q \big(\lambda_b(r) - \lambda_d(r)\big)\, \d r, \qquad q \in [0,1],
\end{equation*}
the net effect of births and deaths on $F^n(t,x)$ is $\approx m^n_t\, f\!\lb(F^n(t,x)/m^n_t\rb)$ per unit time. In particular, up to a Riemann-sum error of order $1/n$, the total mass evolves at rate $f(1) = r_b - r_d$, where $r_b = \int_0^1 \lambda_b$ and $r_d = \int_0^1 \lambda_d$ are the aggregate birth and death rates; and the normalized CDF sees the tilted nonlinearity
\begin{equation}\label{eq:f-tilde}
    \tilde f(q) \;=\; f(q) - q f(1),
\end{equation}
which satisfies $\tilde f(0) = \tilde f(1) = 0$ automatically. The second term in \eqref{eq:f-tilde} is a \emph{dilution} effect: when the population grows, every existing firm's rank is pushed down slightly by the newcomers.

\subsection{The Mean-Field Limit}

\begin{theorem}\label{thm:mean-field}
Assume $\lambda_b, \lambda_d, b \in C^1([0,1])$ and $\sigma^2 \in C^1([0,1])$ with $\min_{[0,1]} \sigma^2 > 0$, and let the initial positions be independent samples from a distribution on $(0,\infty)$ with continuous CDF $u_0$ and $\E|\log X^1_0| < \infty$. Set $m(t) := e^{f(1)t}$. Then for every $T > 0$,
\begin{equation}\label{eq:unifconv}
    \sup_{t \leq T}\, \Big( \big|m^n_t - m(t)\big| + \sup_{x > 0}\, \big| F^n(t,x) - m(t)\, w(t,\log x) \big| \Big) \;\longrightarrow\; 0
    \quad \text{in probability as } n \to \infty,
\end{equation}
where $w$ is the unique CDF solution (in the sense of Lemma~\ref{lem:uniqueness} below), in the log-capitalization variable $y$, of the quasilinear reaction-diffusion equation
\begin{equation}\label{eq:pde-general}
    \pp_t w \;=\; \pp^2_{yy} A(w) - \pp_y B(w) + \tilde f(w),
    \qquad w(0, y) = u_0(e^y),
\end{equation}
with
\begin{equation*}
    A(q) \;=\; \int_0^q \frac{\sigma^2(r)}{2}\,\d r, \qquad
    B(q) \;=\; \int_0^q \tilde b(r)\,\d r, \qquad
    \tilde b \;=\; b - \frac{\sigma^2}{2}.
\end{equation*}
\end{theorem}

Throughout the paper, $w$ denotes the normalized CDF in the log-capitalization variable $y$, and $u$ its expression in the capitalization variable $x = e^y$. The chain rule gives:

\begin{corollary}\label{prop:constant}
If the drift and volatility are constant, $b \equiv \mu$ and $\sigma(\cdot) \equiv \sigma$, then $u(t,x) := w(t, \log x)$ solves, in the original capitalization variable,
\begin{equation}\label{eq:pde}
    \pp_t u \;=\; \frac{\sigma^2 x^2}{2}\,\pp^2_{xx} u + (\sigma^2 - \mu)\, x \,\pp_x u + \tilde f(u),
    \qquad u(0, \cdot) = u_0.
\end{equation}
\end{corollary}

The limit in Theorem \ref{thm:mean-field} splits into two parts. In the limit, the number of firms grows (or shrinks) deterministically, following the mass curve $m(t) = e^{f(1)t}$. The \emph{shape} of the distribution, described by the normalized CDF $w$, evolves by \eqref{eq:pde-general}: the differential operator is the one satisfied by the CDF of a diffusion with rank-dependent coefficients, as in \cite{shkolnikov_large_2012, jourdain_propagation_2013}, and the reaction term $\tilde f(w)$ adds or removes mass according to the net local intensity of entry and exit at rank $w$. In the constant-coefficient case, the operator reduces to the familiar one for geometric Brownian motion; this is the version we validate numerically in Appendix~\ref{sec:validation}.

\begin{remark}
In the constant-coefficient case, the operator in \eqref{eq:pde} is not the Fokker--Planck operator in divergence form, because $u$ is a CDF rather than a density. One obtains it by integrating the Fokker--Planck equation for the density $\rho = \pp_x u$ over $(0, x]$:
\begin{equation*}
    \pp_t u \;=\; \int_0^x \Big[ - \pp_z(\mu z \rho) + \tfrac{1}{2}\pp^2_{zz}(\sigma^2 z^2 \rho) \Big] \,\d z
    \;=\; -\mu x \rho + \tfrac{1}{2}\pp_x(\sigma^2 x^2 \rho)
    \;=\; \frac{\sigma^2 x^2}{2} \pp^2_{xx} u + (\sigma^2 - \mu)x \pp_x u.
\end{equation*}
\end{remark}

\begin{remark}\label{rem:immigration}
In the particle system, a new firm is attached to a parent and born at the parent's position. One could instead let new firms enter at an independent position drawn from the current rank profile, which is closer to how listings actually work. Both mechanisms produce the same rate of births per rank interval, so they have the same mean-field limit.
\end{remark}

\subsection{Proof of Theorem \ref{thm:mean-field}}
\label{subsec:proof}

Throughout this subsection, we work in log coordinates $Y^i_t = \log X^i_t$, and It\^o's formula turns each particle into a diffusion with rank-dependent drift $\tilde b(q_i) = b(q_i) - \sigma^2(q_i)/2$ and diffusion coefficient $\sigma(q_i)$. The ranks are preserved because $\log$ is increasing. Write $\nu^n_t$ for the scaled empirical measure of the log-positions, and $G^n(t,y) = \nu^n_t((-\infty, y])$ for its scaled CDF. 

The proof follows the usual compactness argument for mean-field limits. We write the empirical measure as a semimartingale, bound its moments, and prove tightness. We then pass to the limit, identify the limit with the PDE, and prove uniqueness. Our new analysis is concentrated in four lemmata, which we state first. The first two turn weak convergence of scaled CDFs into uniform convergence: a fixed-time P\'olya lemma, and a uniform-in-time version. These results are needed for the conclusion of the theorem (Step 6 below) to be uniform in rank and time. Our third and fourth supporting lemmata give the uniqueness, and the existence with regularity, of CDF solutions of the limiting equation; uniqueness enters Step 5, and the regularity is what the uniform P\'olya lemma uses in Step 6. Steps 1 through 5 prove the law of large numbers in the weak topology of measure-valued paths, the mode of convergence in \cite{shkolnikov_large_2012, jourdain_propagation_2013}. Our applications below require uniform convergence. The objects of Sections~\ref{sec:wave} and \ref{sec:diversity} live at the top of the market: the capital distribution curve over the top ranks, and portfolios weighted by powers of capitalization. Uniform convergence in rank is what makes the model's statements at the leading edge meaningful.

We present our first P\'olya lemma:
\begin{lemma}\label{lem:polya}
Let $(H_n)$ be scaled CDFs (nondecreasing, right-continuous, $0 \leq H_n \leq C$) converging pointwise and at $\pm\infty$ to a continuous nondecreasing bounded $H$. Then $\sup_y |H_n(y) - H(y)| \to 0$.
\end{lemma}

\begin{proof}
Partition the range of $H$ into $k$ equal pieces by points $y_1 < \cdots < y_{k-1}$ (possible since $H$ is continuous). For $y \in [y_j, y_{j+1}]$, monotonicity gives $H_n(y) - H(y) \leq H_n(y_{j+1}) - H(y_j) \leq H_n(y_{j+1}) - H(y_{j+1}) + C/k$, and similarly from below. For $y \leq y_1$, monotonicity gives $H_n(y) - H(y) \leq H_n(y_1) - H(-\infty)$ and $H(y) - H_n(y) \leq H(y_1) - H_n(-\infty)$, which the convergence at $-\infty$ and the grid make small. The tail above $y_{k-1}$ follows symmetrically, using $H_n(+\infty) \to H(+\infty)$. Taking $n \to \infty$ and then $k \to \infty$ proves the claim.
\end{proof}

Compare \cite[Lemma 2.11]{van_der_vaart_asymptotic_1998} for the case of probability CDFs, and a similar trick for monotone functions in the Skorokhod space is given by \cite[Corollaries 12.5.1 and 12.11.1]{whitt_stochastic-process_2002}. 
Our second lemma is the uniform-in-time version of the first. The hypothesis changes in two ways: the pointwise evaluations of the scaled CDFs are replaced by pairings with compactly supported continuous functions, and both these pairings and the masses are required to converge uniformly in time. The proof fixes a finite grid in advance, using the joint continuity of the limit for the spacing and Dini's theorem for the tails, and controls each grid point by a pair of test functions.

\begin{lemma}\label{lem:polya-uniform}
For each $t \in [0,T]$ let $\xi_t$ be the finite measure on $\R$ with scaled CDF $H(t,y) = m_t\, V(t,y)$, where $m$ is continuous and positive and $V$ is jointly continuous, nondecreasing in $y$, with $V(t, -\infty) = 0$ and $V(t, +\infty) = 1$ for every $t$. Let $(\xi^n_t)_{t \leq T}$ be paths of finite measures with scaled CDFs $H_n(t,y) = \xi^n_t((-\infty, y])$ and masses $m^n_t = \xi^n_t(\R)$. If
\begin{equation*}
    \sup_{t \leq T} \big| \langle \xi^n_t, \varphi \rangle - \langle \xi_t, \varphi \rangle \big| \;\longrightarrow\; 0
    \quad \text{for every } \varphi \in C_c(\R),
    \qquad \text{and} \qquad
    \sup_{t \leq T} \big| m^n_t - m_t \big| \;\longrightarrow\; 0,
\end{equation*}
then $\sup_{t \leq T} \sup_{y \in \R} \big| H_n(t,y) - H(t,y) \big| \to 0$.
\end{lemma}

\begin{proof}
Fix $\varepsilon > 0$. We first truncate the domain: we choose $R$ so that, uniformly in time, the limit carries mass at most $\varepsilon$ outside $[-R, R]$, and then work on a finite grid inside the truncated window. The functions $t \mapsto H(t, -R)$ and $t \mapsto m_t - H(t, R)$ are continuous, nonincreasing in $R$, and vanish pointwise as $R \to \infty$, so by Dini's theorem (applied along integer $R$, using monotonicity in $R$) there is an $R$ with $\sup_{t \leq T} H(t, -R) < \varepsilon$ and $\sup_{t \leq T} \big( m_t - H(t, R) \big) < \varepsilon$; and $H$ is uniformly continuous on the compact set $[0,T] \times [-R-1, R+1]$, so there are points $-R = y_0 < \cdots < y_k = R$ with $\sup_{t \leq T} \big( H(t, y_{j+1}) - H(t, y_j) \big) < \varepsilon$ for every $j$. At each interior grid point, we squeeze from below and above with two continuous compactly supported functions,
\begin{equation*}
    \ind_{[-R,\, y_{j-1}]} \;\leq\; \varphi_j^- \;\leq\; \ind_{[-R-1,\, y_j]},
    \qquad
    \ind_{[y_{j+1},\, R]} \;\leq\; \varphi_j^+ \;\leq\; \ind_{[y_j,\, R+1]},
\end{equation*}
so that $\langle \xi^n_t, \varphi_j^- \rangle \leq H_n(t, y_j) \leq m^n_t - \langle \xi^n_t, \varphi_j^+ \rangle$, while $\langle \xi_t, \varphi_j^- \rangle \geq H(t, y_j) - 2\varepsilon$ and $m_t - \langle \xi_t, \varphi_j^+ \rangle \leq H(t, y_j) + 2\varepsilon$ uniformly in $t$, by the choice of $R$ and of the grid. The hypotheses, applied to the finitely many functions $\varphi_j^\pm$, therefore give
\begin{equation*}
    \limsup_{n}\, \max_{j}\, \sup_{t \leq T} \big| H_n(t, y_j) - H(t, y_j) \big| \;\leq\; 3\varepsilon
\end{equation*}
(at the two endpoint grid points the tails give the bound directly). Between grid points, monotonicity in $y$ sandwiches exactly as in Lemma~\ref{lem:polya}; below $-R$ and above $R$ the tails and the masses control both functions. Putting everything together yields $\limsup_n \sup_{t \leq T} \sup_y |H_n(t,y) - H(t,y)| \leq 5\varepsilon$, and we conclude since $\varepsilon$ was arbitrary.
\end{proof}

Our third and fourth preliminary results concern the limiting PDE itself, and are where the assumption $\min \sigma^2 > 0$ enters: uniqueness of CDF solutions, by a duality argument, and existence with regularity for Step 6. The proofs are collected in Appendix~\ref{app:pde-proofs}. The corresponding uniqueness statement for the possibly degenerate case without reaction term is \cite[Proposition 2.2]{jourdain_propagation_2013}.

\begin{lemma}[Uniqueness of CDF solutions]\label{lem:uniqueness}
Call $w$ an admissible solution of \eqref{eq:pde-general} on $[0,T]$ if $w(t, \cdot)$ is the CDF of a probability measure $\pi_t$ with finite first moment, $t \mapsto \pi_t$ is weakly continuous, $\int_0^T \!\!\int_\R |y|\, \pi_t(\d y)\, \d t < \infty$, and $w$ satisfies \eqref{eq:pde-general} in the sense of distributions. Two admissible solutions with the same initial CDF coincide.
\end{lemma}

\begin{lemma}[Existence and regularity of the CDF solution]\label{lem:regularity}
Let $w_0(y) = u_0(e^y)$ be a continuous CDF with finite first moment. Then \eqref{eq:pde-general} has an admissible solution $w$ in the sense of Lemma~\ref{lem:uniqueness}, with $w$ and $\pp_y w$ continuous for $t > 0$. Moreover, $w(t, \cdot)$ is a continuous CDF for every $t$, and the map $t \mapsto w(t, \cdot)$ is continuous in the supremum norm; in particular $w$ is jointly continuous.
\end{lemma}

Everything is in place to proceed with the proof of the theorem. We note that the two halves of \eqref{eq:unifconv} are proved in different places: the mass at the end of Step 2, and the shape part in Step 6.

\begin{proof}[Proof of Theorem~\ref{thm:mean-field}]

\textbf{Step 1: Semimartingale Decomposition.}
For a test function $\varphi \in C^2_c(\R)$, the process $t \mapsto \langle \nu^n_t, \varphi\rangle$ moves in two ways: continuously, as the particles diffuse, and by jumps of size $\pm\varphi(Y^i_t)/n$ at each birth or death. Compensating each jump by its intensity produces the reaction integral below, and what remains (Brownian fluctuations plus compensated jumps) is a martingale:
\begin{equation}\label{eq:semimartingale}
    \langle \nu^n_t, \varphi \rangle \;=\; \langle \nu^n_0, \varphi \rangle
    + \int_0^t \big\langle \nu^n_s, \mathcal{A}_{G^n_s/m^n_s}\varphi \big\rangle \,\d s
    + \int_0^t \big\langle \nu^n_s, \big(\lambda_b - \lambda_d\big)\big(G^n_s/m^n_s\big)\,\varphi \big\rangle \,\d s
    + M^n_t(\varphi),
\end{equation}
where $\mathcal{A}_{H}\varphi(y) = \frac{\sigma^2(H(y))}{2}\varphi''(y) + \tilde b(H(y))\,\varphi'(y)$ for a rank profile $H$, and $M^n(\varphi)$ is a square-integrable martingale. Compositions such as $(\lambda_b - \lambda_d)(G^n_s/m^n_s)$ abbreviate the function $y \mapsto \lambda_b(G^n_s(y)/m^n_s) - \lambda_d(G^n_s(y)/m^n_s)$. This is the usual construction; see \cite{graham_asymptotic_1996} for the general measure-valued framework and \cite{shkolnikov_large_2012, jourdain_propagation_2013} for the rank-dependent part without births/deaths.

\paragraph{Step 2: Moment and Martingale Bounds.}
Each particle's total event rate is at most the constant $\Lambda = \sup\lambda_b + \sup\lambda_d$, so the population can be coupled with a pure birth process (also known as a Yule process) $N_t$ of per-individual rate $\Lambda$ started from $N_0 = n$. Give each individual of the comparison process an independent Poisson clock of rate $\Lambda$, let it branch at every ring, match each living particle of our system injectively to a compared individual, and let a matched particle give birth only at rings of its partner's clock, accepting each ring with probability $\lambda_b(q_i)/\Lambda \leq 1$; deaths remove particles from our system only. The coupling maintains $n_t \leq N_t$, and the birth process is finite for all times with $\E[N_t] = n e^{\Lambda t}$ \cite[Section 2.5]{norris_markov_1997}. In particular the particle system does not explode, and $\E[m^n_t] \leq e^{\Lambda t}$.

Next we prove a moment bound, used in the cutoff argument below and again in Steps 3 through 5. Apply \eqref{eq:semimartingale} to smooth compactly supported $\psi_K \uparrow |y|$ with $|\psi_K'| \leq 1$ and $\sup_K \|\psi_K''\|_\infty < \infty$. The diffusive term is bounded by $C\, m^n_s$, since the coefficients are bounded. The reaction term is bounded by $\Lambda\, \langle \nu^n_s, \psi_K \rangle$, since a birth duplicates the position of a living particle. Taking expectations removes the martingale, and Gronwall's inequality gives $\sup_{t \leq T} \E\langle \nu^n_t, \psi_K \rangle \leq C(T)\big(1 + \E|\log X^1_0|\big)$, uniformly in $K$. Monotone convergence in $K$ then yields $\sup_{t \leq T} \E\langle \nu^n_t, |y| \rangle < \infty$.

The quadratic variation of the martingale part is small because every jump is small: births and deaths arrive one particle at a time, so each event moves $\langle \nu^n, \varphi \rangle$ by at most $\|\varphi\|_\infty / n$ and arrives at rate at most $\Lambda$ per living particle, while each particle's Brownian motion contributes to $\langle \nu^n, \varphi \rangle$ at the rate $\sigma^2(q_i)\,\varphi'(Y^i)^2/n^2$. Summing the two contributions over the at most $N_t$ living particles and using $\E[N_t] = n e^{\Lambda t}$ gives
\begin{equation*}
    \E\big[\langle M^n(\varphi) \rangle_t\big] \;\leq\; \frac{C(\varphi, T)}{n}, \qquad t \leq T.
\end{equation*}
Doob's inequality then gives $\sup_{t \leq T} |M^n_t(\varphi)| \to 0$ in $L^2$. To pass \eqref{eq:semimartingale} to $\varphi \equiv 1$, which is not compactly supported, take cutoffs $\varphi_R$ with $\varphi_R = 1$ on $[-R, R]$, $0 \leq \varphi_R \leq 1$, and $|\varphi_R'| + |\varphi_R''| \leq C/R$; the moment bound above shows the diffusive terms vanish as $R \to \infty$, while ranks are unaffected.

In the limit $R \to \infty$, since the ranks $G^n/m^n$ sweep the uniform grid, the reaction term reduces to $\int_0^t m^n_s f(1)\,\d s$ up to an error of order $t/n$. To make things explicit: the reaction integrand is the Riemann sum $\frac{1}{n} \sum_{k=1}^{n_s} (\lambda_b - \lambda_d)(k/n_s) = m^n_s\, f(1) + O(1/n)$, where the error is $O(1/n)$ because $\lambda_b - \lambda_d$ is Lipschitz and $m^n_s / n_s = 1/n$. With $\varphi \equiv 1$, \eqref{eq:semimartingale} therefore reads $m^n_t = 1 + f(1) \int_0^t m^n_s \,\d s + O(t/n) + M^n_t(1)$. The limit mass $m(t) = e^{f(1)t}$ satisfies the same equation with no error and no martingale, so subtracting the two equations and applying Gronwall's inequality gives
\begin{equation*}
    \sup_{t \leq T} \big| m^n_t - m(t) \big| \;\leq\; \Big( \frac{C T}{n} + \sup_{t \leq T} \big| M^n_t(1) \big| \Big)\, e^{|f(1)|\, T} \;\longrightarrow\; 0 \quad \text{in probability},
\end{equation*}
the mass part of \eqref{eq:unifconv}. The convergence in probability holds since the bound is pathwise: the term $CT/n$ is deterministic, and the martingale supremum vanishes in $L^2$ and hence in probability.

\paragraph{Step 3: Tightness.}
Tightness of the laws of $(\nu^n)_n$ in the Skorokhod space of measure-valued paths follows from the bounds of Step 2 and the Aldous--Rebolledo criterion, following the pattern of \cite[Section 4]{graham_asymptotic_1996}: the mass bound $\E[\sup_{t \leq T} m^n_t] \leq e^{\Lambda T}$ from the Yule coupling controls the total masses, the moment bound gives compact containment, and the drift and reaction integrands in \eqref{eq:semimartingale} are bounded on bounded time intervals while the martingale part vanishes. Hence the laws are tight in the weak topology and no mass escapes to infinity.

\paragraph{Step 4: Passing to the Limit.}
Along a convergent subsequence $\nu^n \Rightarrow \nu$, we must show the nonlinear term in \eqref{eq:semimartingale} converges to the same expression with $G^n_s/m^n_s$ replaced by its limit. The composition $\lambda(G^n)$ is not continuous in the topology of weak convergence, so we will move the composition inside an integral, where pointwise convergence suffices. For a continuous function $h$ on $[0,1]$ and a mass $m > 0$, write $H_m(q) = \int_0^q h(r/m)\, \d r$. The three instances we need are $h = \sigma^2/2$, $h = \tilde b$, and $h = \lambda_b - \lambda_d$; and we write $A_m$, $B_m$, and $m f(\cdot / m)$ for the corresponding $H_m$.

We make two observations: the first rewrites each pairing at the particle level, and the second passes the rewritten form to the limit. \emph{First observation.} At the particle level each pairing is a Stieltjes integral up to $O(1/n)$: for $\psi \in C^1_c(\R)$,
\begin{equation}\label{eq:stieltjes}
    \big\langle \nu^n_s,\, h\big(G^n_s/m^n_s\big)\, \psi \big\rangle
    \;=\; \int_\R \psi(y)\, \d_y \big[ H_{m^n_s}\big(G^n_s(y)\big) \big] + O(1/n)
    \;=\; -\int_\R \psi'(y)\, H_{m^n_s}\big(G^n_s(y)\big)\, \d y + O(1/n),
\end{equation}
because the function $y \mapsto H_{m^n_s}(G^n_s(y))$ jumps by $\frac{1}{n}\, h(\tilde q_i)$ at the $i$-th particle, for some $\tilde q_i$ within $1/n_s$ of the rank $q_i(s)$, so the total error is at most $\|\psi\|_\infty\, \mathrm{Lip}(h)\, (n_s/n)(1/n_s) = \|\psi\|_\infty\, \mathrm{Lip}(h)/n$: the same cancellation as in Step 2. This moves the composition out of the measure and into a bounded integrand.

\emph{Second observation.} The right side of \eqref{eq:stieltjes} is stable under weak convergence, even in the presence of atoms. To argue pointwise along the subsequence, we use Skorokhod's representation theorem \cite[Theorem 3.2.2]{whitt_stochastic-process_2002}: on a new probability space, versions of the paths converge almost surely in the Skorokhod topology. We may take a further subsequence so that the mass convergence of Step 2 also holds almost surely (the masses are functionals of the paths). Along the representation, $\sup_n \sup_{s \leq T} m^n_s$ is bounded, and almost surely the marginals converge at every continuity time of the limit path, a co-countable set. At almost every $s$, we have $\nu^n_s \Rightarrow \nu_s$, so $G^n_s(y) \to G_s(y)$ at every continuity point of $G_s$, hence at almost every $y$; $m^n_s \to m_s \geq e^{-\Lambda T} > 0$ by Step 2; and $H_m(q)$ is bounded and jointly continuous in $(m, q)$. So $H_{m^n_s}(G^n_s(y)) \to H_{m_s}(G_s(y))$ for almost every $y$ and almost every $s$, and dominated convergence passes the right side of \eqref{eq:stieltjes} to the limit. Combining the two observations, with $\psi = \varphi'', \varphi', \varphi$ for $\varphi \in C^3_c(\R)$, passes every nonlinear term in \eqref{eq:semimartingale} to the limit. No continuity of the limit is required.

\paragraph{Step 5: Identification and Uniqueness.}
Almost surely along the representation of Step 4, every limit point satisfies the weak form of
\begin{equation}\label{eq:log-pde}
    \pp_t v \;=\; \pp^2_{yy} A_{m_t}(v) - \pp_y B_{m_t}(v) + m_t\, f(v / m_t), \qquad v(0,\cdot) = u_0(\exp(\cdot)),
\end{equation}
for the limiting scaled CDF $v(t,y) = G_t(y)$, with $A_m$ and $B_m$ as in Step 4. Substituting $v = m w$ with $\dot m = f(1)m$ and using the homogeneity $A_m(m w) = m A(w)$, $B_m(m w) = m B(w)$ shows that the normalized function $w = v/m$ satisfies exactly \eqref{eq:pde-general}. Any normalized limit point is admissible in the sense of Lemma~\ref{lem:uniqueness}: its marginals are probability measures, weakly continuous in $t$ (their action on test functions is a time integral of bounded integrands), and the moment condition holds by the moment bound of Step 2 and Fatou's lemma along the marginal convergence, at almost every $t$ directly and at every $t$ by weak continuity, since $\pi \mapsto \langle \pi, |y| \wedge K \rangle$ is weakly continuous and monotone in $K$. Lemma~\ref{lem:uniqueness} then identifies every limit point with the solution $w$ of Lemma~\ref{lem:regularity}. Subsequential convergence upgrades to convergence of the full sequence: $\nu^n \Rightarrow \nu$ in the Skorokhod space, where $\nu$ is the path of measures with scaled CDFs $H(t, \cdot) := m(t)\, w(t, \cdot)$.

\paragraph{Step 6: Uniform Convergence.}
It remains to upgrade the topology to uniform convergence. By Lemma~\ref{lem:regularity}, $w$ is jointly continuous and $w(t, \cdot)$ is a continuous CDF for every $t$, so the limit $H(t,y) = m(t)\, w(t,y)$ satisfies the continuity hypotheses of Lemma~\ref{lem:polya-uniform}. Since the limit is deterministic, convergence in probability of the supremum $S_n := \sup_{t \leq T} \sup_y |G^n(t,y) - H(t,y)|$ follows if every subsequence admits a further subsequence along which versions of $S_n$ converge to $0$ almost surely; and $S_n$ is a functional of the path $\nu^n$, so a Skorokhod representation does not change its law. Fix a subsequence and apply a Skorokhod representation with the mass refinement, exactly as in Step 4; the limit is now the deterministic path $\nu$ of Step 5. Almost surely, then, Skorokhod convergence to the time-continuous path $\nu$ is uniform in the metric of weak convergence; a countable family of test functions metrizes that topology on sets of bounded mass, and the masses are bounded along the refined representation, so the uniformity extends to $\sup_{t \leq T} |\langle \nu^n_t, \varphi \rangle - \langle \nu_t, \varphi \rangle| \to 0$ for every $\varphi \in C_c(\R)$, and Lemma~\ref{lem:polya-uniform} yields
\begin{equation*}
    \sup_{t \leq T}\, \sup_{y \in \R}\, \big| G^n(t,y) - m(t)\, w(t,y) \big| \;\longrightarrow\; 0,
\end{equation*}
which is the statement of the theorem, in probability, after undoing the substitution $y = \log x$.
\end{proof}

\begin{remark}[Bounded intensities]\label{rem:bounded}
The proof of Theorem \ref{thm:mean-field} uses $\sup \lambda_b + \sup \lambda_d < \infty$ in two places: the Yule coupling of Step 2, which dominates the population by a pure birth process, and the quadratic-variation bound of the same step, where every jump is of size at most $\|\varphi\|_\infty/n$ and arrives at a bounded rate. This is why we require the fitted intensities in Section~\ref{sec:empirics} to be bounded functions of rank, and it is what rules out the otherwise convenient and parsimonious Beta distributions for $\lambda_b$ and $\lambda_d$. As for regularity, the proof uses only the Lipschitz property of $\lambda_b - \lambda_d$, in the Riemann-sum error of Step 2 and the Stieltjes error of Step 4.
\end{remark}

\begin{remark}[Degenerate diffusions]\label{rem:atlas}
The uniform ellipticity assumption $\min \sigma^2 > 0$ keeps every step of the proof inside standard uniformly parabolic theory. Empirically, the volatility curve measured from returns is bounded well away from zero at every rank (see Subsection~\ref{subsec:ranked-vol}). The assumption is economically necessary: in the finite system, a firm's rank is constant between crossings, so a rank with zero volatility would make its capitalization locally deterministic; unless its growth rate equals the risk-free rate, this is an arbitrage against the money market. We note that the degenerate variants in the rank-based literature place the degeneracy in the limiting equation rather than in any single firm's volatility: in the porous-medium limits of \cite{shkolnikov_large_2012, jourdain_propagation_2013} the diffusivity of the limit may vanish, and volatility-stabilized markets, in which small-firm volatility blows up rather than vanishes, admit an explicit degenerate mean-field limit \cite{shkolnikov_large_2013}. Within our proof, ellipticity enters in two places: the uniqueness of CDF solutions with a source term (Lemma~\ref{lem:uniqueness}) and the regularity of the limit (Lemma~\ref{lem:regularity}). The identification of Step 4 survives degeneracy (the rewriting via $A_m$, $B_m$ uses only continuity and boundedness of the coefficients), so in a degenerate variant Steps 1 through 5 would still give a law of large numbers in the weak topology, given a substitute proof of uniqueness; only the uniform statement needs the regularity of the limit. In the drift-stabilized models the analogous regularity is assumed or inherited from stationary initial spacings \cite[Theorem 1.2]{shkolnikov_large_2012}, and degenerate diffusion without reaction requires the delicate duality argument of \cite[Proposition 2.2 and Appendix A]{jourdain_propagation_2013}.
\end{remark}

\section{Entry and Exit in CRSP Data}
\label{sec:empirics}

We now estimate the two intensity functions from historical data. The takeaway from this section is that both intensities are strongly rank-dependent: exit intensity is an order of magnitude higher for the smallest firms than in the middle of the distribution, and entering firms arrive predominantly in the lower half.

\subsection{Data Construction}
\label{subsec:data}

We use the CRSP Monthly Stock File from \SampleFirst{} to \SampleLast{}, obtained through Wharton Research Data Services, together with the CRSP names, delisting, and distribution event files. We restrict the universe to common shares (share codes 10 and 11) listed on the NYSE, AMEX, or NASDAQ (exchange codes 1--3), which removes exchange-traded funds, closed-end funds, American depositary receipts, and similar securities whose listings and delistings are not firm entries or exits. Because a firm may list several share classes, we aggregate market capitalization (price times shares outstanding) to the firm level using CRSP's permanent company identifier (\emph{permco}). This leaves \NFirms{} firms and \NFirmMonths{} firm-months, on average \AvgFirms{} listed firms per month. We note that CRSP covers only the NYSE before its two universe expansions: AMEX enters the file in July 1962 and NASDAQ in December 1972, each bringing thousands of already-existing firms at once. The early decades of the sample therefore describe a narrower market, and Subsection~\ref{subsec:growth-acct} separates this coverage growth from economic entry.

Each month, we rank firms by market capitalization and record each firm's rank percentile $q \in (0, 1]$, with $q = 1$ the largest firm. We call a month's roster of firms and their capitalizations that month's \emph{cross-section}. Our convention matches the orientation of the CDF in Theorem~\ref{thm:mean-field}: $q$ is the fraction of firms with capitalization at most one's own. Ranks are uniformly distributed across firms within each month by construction, which is what allows us to read the intensity $\lambda(q)$ directly off the rank distribution of events: with ranks uniform within every cross-section, the density of event ranks at $q$ is proportional to the per-firm intensity at $q$.

\paragraph{Exits.} We take exits from the CRSP delisting file, which records a delisting code for every security that stops trading. Codes 200--299 are mergers and acquisitions, 400--499 liquidations, and 500--599 delistings for cause (``dropped'' means bankruptcy, failure to meet listing requirements, or similar). We date a firm's exit at its last month with a valid price and rank it by its percentile in that month. Delisting codes 300--399 (exchanges of securities) and delistings that carry a successor identifier are reorganizations rather than exits (the firm continues under a new name or identifier) so we remove them, together with the paired ``entrance'' of the successor identifier; this removes \NTechDropped{} spurious events. Working from delisting codes rather than from the raw first and last dates of each data series is important since about \PctExitRawTechnical\% of the apparent exits at the ends of the raw series are not economic exits. On delisting-related biases in CRSP see also \cite{shumway_delisting_1997}, and the course notes with accompanying GitHub repository \cite{ruf_introduction_nodate}.

\paragraph{Entrances.} A firm enters when its permco first appears in the filtered universe and we rank it by its percentile in that first month. We flag as spin-offs entrants whose security was distributed to shareholders of an existing firm (CRSP distribution codes 37xx with an acquiring \emph{permno}, the security-level identifier distinct from the firm-level permco), \NEntSpin{} of \NEnt{} entrances. We drop the first and last month of the sample, where beginnings and ends of data series are artifacts of the sample window, and the two universe expansion months.

The cleaned sample contains \NEnt{} entrances and \NExit{} exits (Table~\ref{tab:events}). The  difference is the net growth of the listed universe, about \NetGrowthAnnualPct\% per year on average. In the model this imbalance is exactly the mass growth rate $f(1) = r_b - r_d$ of Theorem~\ref{thm:mean-field}, so no balancing adjustment is needed.

\begin{table}[t]
\centering
\caption{Cleaned entry and exit events, CRSP \SampleFirstYear--\SampleLastYear.}
\label{tab:events}
\begin{tabular}{llr}
\toprule
Event & Class & Count \\
\midrule
Entrance & new listing & \NEntNew \\
         & spin-off    & \NEntSpin \\
\midrule
Exit & dropped (bankruptcy, delisted for cause) & \NExitDropped \\
     & merger or acquisition                    & \NExitMerger \\
     & liquidation                              & \NExitLiq \\
\bottomrule
\end{tabular}
\end{table}

\paragraph{Mergers.}
Mergers are \PctExitMerger\% of all true exits, so how they are treated matters. We count the merger target as a death: the firm genuinely leaves the population, which is the object our model tracks. What the binary birth/death model does not capture is that the target's capital moves to the acquirer rather than vanishing; in the mean-field limit this would add a transport term to \eqref{eq:pde}. We return to this in Section~\ref{sec:discussion}, and we check below that the shape of the death intensity is not driven by the merger events alone.

\subsection{Fitting the Intensities}
\label{subsec:fitting}

From here on, the estimation window is 1975--2024, the half-century after both universe expansions, so that the sample has stable coverage; the full century returns as a robustness check in Subsection~\ref{subsec:reaction}. Figure~\ref{fig:intensity-fit} shows the rank distributions. Entrances are spread over the lower and middle ranks, with a peak near the 25th percentile: firms rarely enter at the smallest sizes, since a listing requires a minimum scale, and rarely enter at the largest. Exits concentrate heavily at the bottom: \PctExitBottomFive\% of all exits occur in the lowest 5\% of ranks. Figure~\ref{fig:exit-decomp} splits the exits by delisting reason. Delistings for cause dominate at the very bottom, while mergers are spread across the whole rank range. Conditional on a mid-sized or large firm exiting, the exit is almost always a merger. The bottom-heavy shape of the total exit intensity is therefore not a merger artifact.

\begin{figure}[ht]
\centering
\includegraphics[width=\textwidth]{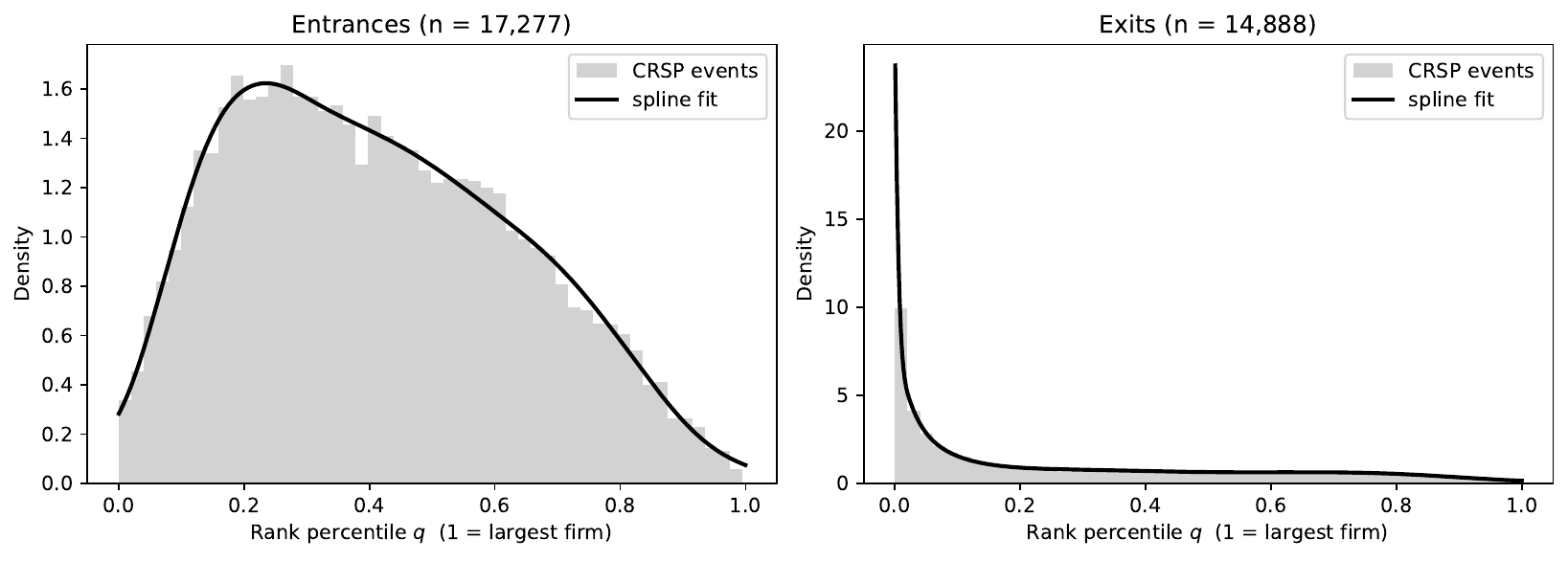}
\caption{Rank distributions of entrance (left) and exit (right) events on the window 1975--2024, with the fitted spline intensities of Subsection~\ref{subsec:fitting}. Rank percentile $q = 1$ is the largest firm.}
\label{fig:intensity-fit}
\end{figure}

\begin{figure}[t]
\centering
\includegraphics[width=\textwidth]{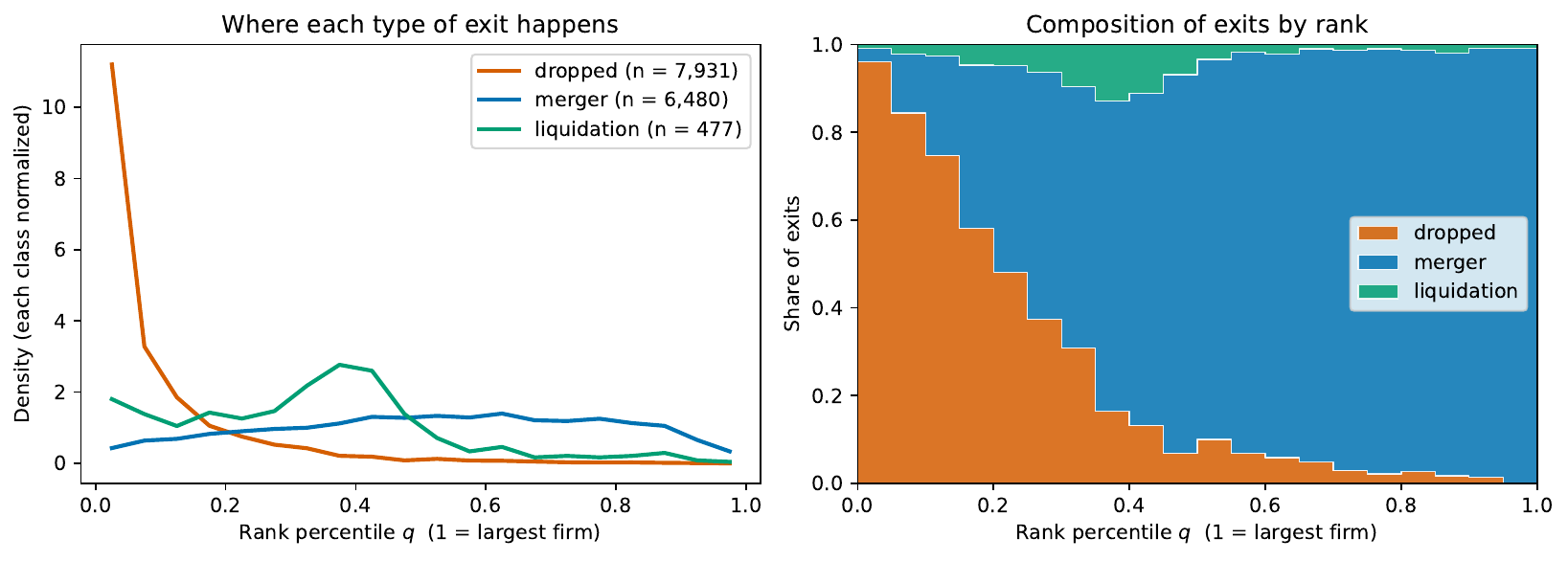}
\caption{Exit events by delisting reason. Left: the rank distribution of each exit type, normalized separately so that shapes can be compared (liquidations are only \NExitLiqWin{} events). Right: the composition of exits across ranks. Failures concentrate at the bottom; among the exits of mid-sized and large firms, mergers dominate.}
\label{fig:exit-decomp}
\end{figure}

Theorem~\ref{thm:mean-field} requires the intensities to be bounded Lipschitz functions of rank, and the wave theory of Section~\ref{sec:relaxation} asks for continuity of $\lambda_b - \lambda_d$ at the endpoints (Remark~\ref{rem:bounded}), so we estimate them by a Poisson regression with a smooth log-intensity rather than by a parametric density family. Our first attempt used Beta densities, which fit the entrance data well but are unbounded at the boundary for the exit data, violating the assumptions of the theorem.

We partition $(0,1]$ into $J = \NBins$ equal bins, count the events in each bin, and model the counts as Poisson with log-mean
\begin{equation*}
    \log \mu_j \;=\; \boldsymbol{\beta}^\top \mathbf{B}(q_j) + \log(N/J),
\end{equation*}
where $\mathbf{B}$ is a natural cubic spline basis in the rank $q_j$ of bin $j$, and the offset $\log(N/J)$ is the expected count under a uniform intensity, so that the coefficients capture only the shape. Because the exit events are so concentrated near $q = 0$, we place the spline knots at quantiles of the observed event ranks rather than uniformly; the number of basis functions (\DfSplineEnt{} for entrances, \DfSplineExit{} for exits) is chosen by Akaike's information criterion (AIC, see \cite{akaike_new_1974}), searched over four to eight degrees of freedom. The fitted intensity $\exp(\boldsymbol{\hat\beta}^\top \mathbf{B}(q))$, scaled to integrate to the aggregate rates $r_b = \RateBirthMonthly \times 10^{-3}$ and $r_d = \RateDeathMonthly \times 10^{-3}$ per firm-month, is bounded and $C^1$ on $[0,1]$ by construction, as Theorem~\ref{thm:mean-field} requires. Rates are monthly throughout the analysis, and we multiply by twelve to annualize.

Three checks support the specification. First, overdispersion: we refit both models as negative binomial with a free dispersion parameter \cite[Chapter 3]{cameron_regression_2013}, and the estimated dispersion is essentially zero (entrance $\hat\alpha = \NBAlphaEnt$, exit $\hat\alpha = \NBAlphaExit$; Pearson $\chi^2/\mathrm{df}$ of \PoisChiDfEnt{} and \PoisChiDfExit{}), so the Poisson likelihood is adequate. Second, fit: the maximum discrepancy between the empirical CDF of event ranks and the fitted model is \NBKSEnt{} for entrances and \NBKSExit{} for exits: the fitted curves in Figure~\ref{fig:intensity-fit} track the data closely, including the sharp rise of the exit intensity at the bottom. Third, mergers: refitting on exits excluding merger targets leaves the bottom-heavy shape unchanged, with \PctExitBottomFiveOrg\% of the remaining exits in the bottom five percent of ranks, against \PctExitBottomFive\% of all exits. Mergers lift the intensity mainly in the middle ranks: the median merger target exits at rank \MedRankMerger, the median delisting for cause at \MedRankDrop.

\subsection{The Calibrated Reaction Term}
\label{subsec:reaction}

From the fitted intensities we form $f(q) = \int_0^q (\lambda_b - \lambda_d)$ and the tilted version $\tilde f(q) = f(q) - q f(1)$ that governs the normalized distribution in Theorem~\ref{thm:mean-field}. Figure~\ref{fig:reaction} shows both. The shape of $\tilde f$ is the central empirical finding of the paper: it is negative on $(0, \FCross)$ and positive on $(\FCross, 1)$, with $\tilde f(0) = \tilde f(1) = 0$ and strictly negative slope at both endpoints. Among the smallest firms, deaths outpace births, and mass leaves the bottom of the distribution; in the upper-middle ranks, entries (and the dilution pushed down from above) outpace exits. In the language of reaction-diffusion equations, $\tilde f$ is a \emph{bistable} nonlinearity with stable states $0$ and $1$ and unstable interior equilibrium at $q \approx \FCross$. The archetypical bistable situation is given by Nagumo's equation which has $\tilde{f}$ cubic (see \cite{nagumo_active_1962, mckean_nagumos_1970} and Example~\ref{ex:nagumo} below). Both our equation and Nagumo's admit traveling-wave solutions, which we demonstrate rigorously in Section~\ref{sec:relaxation}, and then apply to study the empirical capital distribution curve in Sections~\ref{sec:wave} and \ref{sec:diversity}.

\begin{figure}[t]
\centering
\includegraphics[width=0.72\textwidth]{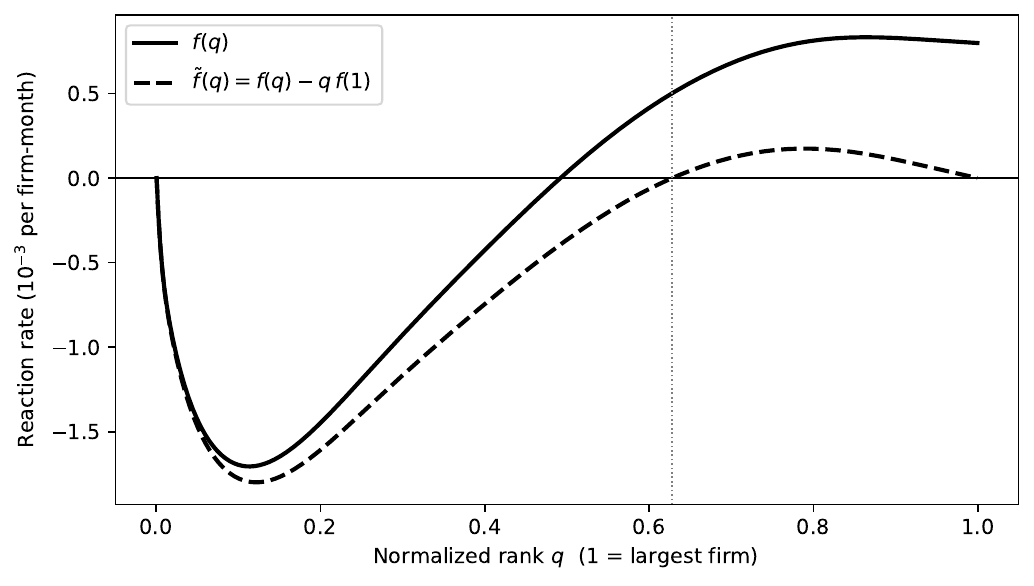}
\caption{The calibrated reaction terms $f$ and $\tilde f = f - q f(1)$, fitted on 1975--2024. The tilted term $\tilde f$ is a bistable nonlinearity: it vanishes at $0$ and $1$ and crosses zero once at $q \approx \FCross$.}
\label{fig:reaction}
\end{figure}

The above calibration pools the post-expansion half-century. Figure~\ref{fig:f-decades} repeats the construction of $\tilde f$ on individual decades. The bistable shape of $\tilde f$ recurs across eras, with one exception: in the 2000s, exits (\TwoThousandsDeathRate\% per year) outpaced entries (\TwoThousandsBirthRate\%) and $\tilde f$ fell below zero at essentially every rank. The levels of activity vary by up to an order of magnitude: decade-average entry rates (that is, $\int_0^1 \lambda_b$) range from \DecBirthRateMin\% to \DecBirthRateMax\% of firms per year, and exit rates from \DecDeathRateMin\% to \DecDeathRateMax\%. The bistable shape is robust to how the century is divided: across ten-year windows starting in 1925 or in 1930, seventeen of nineteen windows show the bistable sign pattern, the two exceptions being the windows containing the early 2000s.
This factorization with stable shapes and strongly time-varying levels motivates the regime-switching extension discussed in Section~\ref{sec:discussion}.
The decade without bistability (the 2000s) leads to a \emph{monostable} reaction-diffusion equation. With $\tilde f \leq 0$ on $(0,1)$ and zeros only at the endpoints, the state $w = 1$ turns unstable while $w = 0$ stays stable. Subsection~\ref{subsec:beyond} records the similarities and differences from the theory of the bistable case.

\begin{figure}[htbp]
\centering
\includegraphics[width=0.8\textwidth]{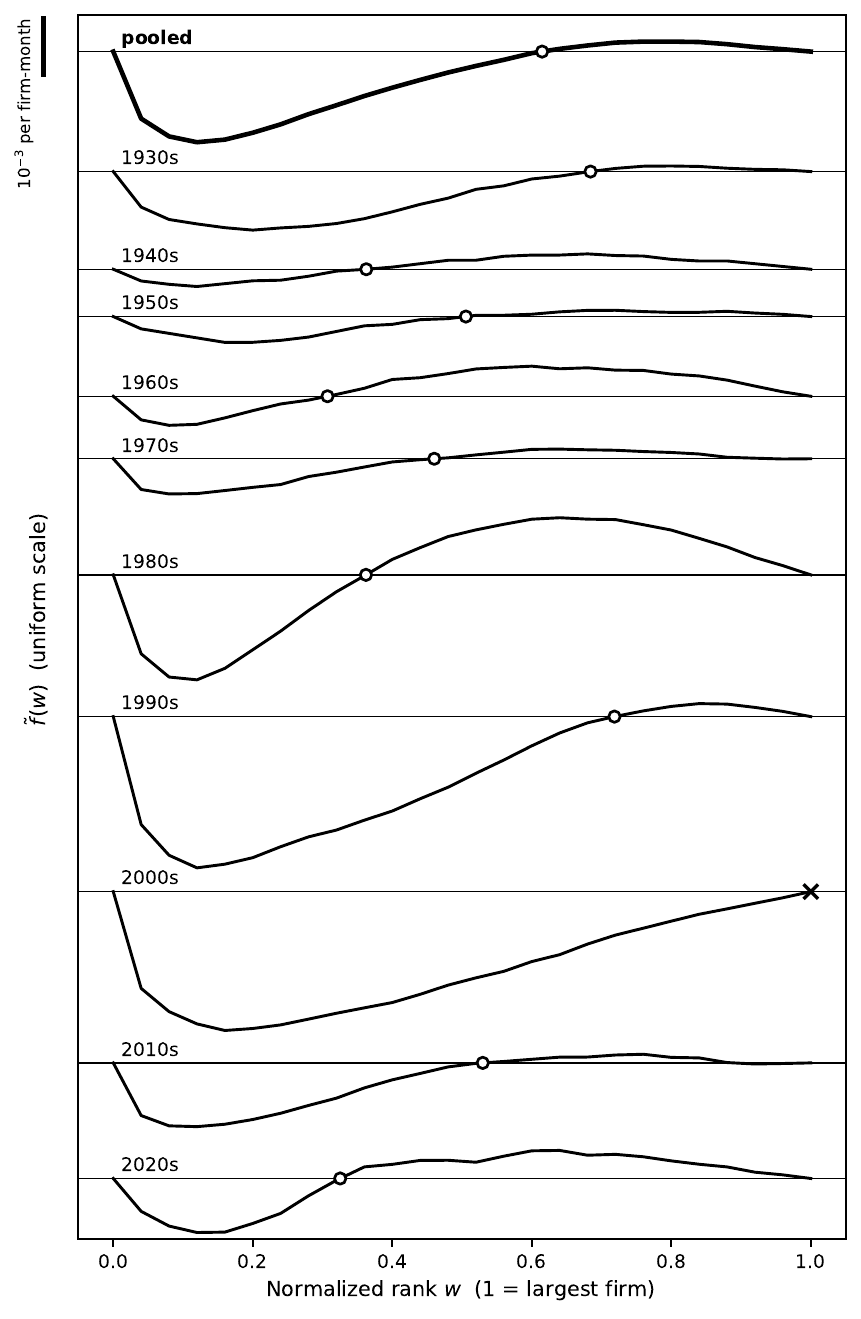}
\caption{The reaction term $\tilde f$ recomputed on each decade (binned estimates; pooled century in bold). Circles mark each decade's interior zero on its axis; the cross marks the 2000s, which has no interior crossing. The bistable shape recurs at varying amplitude; the 2000s lost bistability altogether, with $\tilde f \leq 0$ at almost every rank.}
\label{fig:f-decades}
\end{figure}

The following three functionals of the fitted curves are essential to our analysis in the rest of the paper. Their roles are made precise in Sections~\ref{sec:relaxation} and \ref{sec:wave}, as indicated in parentheses. The first is the interior crossing $q^*$ of $\tilde f$, the unstable equilibrium between the two stable states. The second is the pair of endpoint slopes $\tilde f'(0)$ and $\tilde f'(1)$; in Section~\ref{sec:relaxation} they bound the rate at which the distribution relaxes toward the traveling wave (Proposition~\ref{prop:rate}), and we quote them as the half-lives $\ln 2/|\tilde f'(e)|$. The third is the Pareto tail exponent of the wave, an explicit function of $\tilde f'(1)$ and of the volatility and drift at the top rank (Corollary~\ref{cor:tail}); the fitted curves enter it through the top slope alone, and the volatility and drift enter at their values from Section~\ref{sec:wave} below, measured on the same 1975--2024 window. We quantify the sampling uncertainty of all three by a parametric bootstrap \cite{efron_bootstrap_1979}: we resample the \NBins{} bin counts as Poisson draws from the fitted means, refit the spline with the selected basis held fixed, and rebuild each functional, \BootReps{} times. The bistable sign pattern holds in every replicate. The \BootBandPct\% bands are $[\BootQStarLo, \BootQStarHi]$ for the crossing; \BootHalfBotLoMonths{} to \BootHalfBotHiMonths{} months and \BootHalfTopLoYears{} to \BootHalfTopHiYears{} years for the two half-lives; and $[\BootAlphaLo, \BootAlphaHi]$ for the tail exponent $\alpha = \AlphaRecWave$ of the calibrated wave, the last holding the volatility and drift inputs at their point estimates. For robustness, refitting on the full century of events back to \SampleFirstYear{} gives crossing \CenturyQStar{} and half-lives of \CenturyHalfBotMonths{} months and \CenturyHalfTopYears{} years. 

\section{Relaxation to the Traveling Wave}
\label{sec:relaxation}

The bistable reaction term suggests that the long-run capital distribution should move as a traveling wave. This section develops that structure. Taken together with Theorem~\ref{thm:mean-field}, our results imply that the market model with many firms lies close to a moving wave over long horizons. Subsection~\ref{subsec:wave-identity} derives an identity that any traveling profile must satisfy, coupling the shape to the volatility curve, the drift, and the turnover term. Subsection~\ref{subsec:wave-exist} proves existence and uniqueness of the wave for the full rank-dependent coefficient class, and connects the Pareto tail of the capital distribution to the wave's endpoint behaviour. Subsection~\ref{subsec:relax-sub} proves exponential convergence to the wave for constant $\sigma$ and $\tilde b$, where the equation is semilinear, building on the bistable theory of Fife and McLeod \cite{fife_approach_1977}, and identifies the rate through the spectrum of the linearization about the front (see Proposition~\ref{prop:rate}; the convergence argument for rank-dependent coefficients, which we do not pursue, is outlined in Remark~\ref{rem:rank-dependent}). Subsection~\ref{subsec:beyond} records what carries over for the monostable case, and compares our results to the literature on drift-stabilized models. Throughout, the coefficients satisfy the assumptions of Theorem~\ref{thm:mean-field}, and $\tilde f$ is bistable with $\tilde f'(0) < 0$ and $\tilde f'(1) < 0$; only the monostable case in Subsection~\ref{subsec:beyond} drops bistability.

\subsection{The Wave Identity}
\label{subsec:wave-identity}

The wave identity is formally derived from a nonlinear ordinary differential equation (ODE). Substituting the ansatz $w(t,y) = V(y - ct)$ into \eqref{eq:pde-general}, the time derivative contributes $-cV'$, and a traveling profile must solve the equation
\begin{equation}\label{eq:wave-ode}
    \frac{\d}{\d\xi}\Big(\frac{\sigma^2(V)}{2}\,V'\Big) + \big(c - \tilde b(V)\big)\,V' + \tilde f(V) \;=\; 0, \qquad \xi = y - ct.
\end{equation}
Since $\frac{\sigma^2(V)}{2}V' = \frac{\d}{\d\xi}A(V)$ and $\big(c - \tilde b(V)\big)V' = \frac{\d}{\d\xi}\big(cV - B(V)\big)$, with $A$ and $B$ as in Theorem~\ref{thm:mean-field}, we obtain
\begin{equation*}
    \frac{\d}{\d\xi}\lb( \frac{\d}{\d\xi}A(V) - B(V) + cV \rb) \;=\; -\,\tilde f(V).
\end{equation*}
Integrating from $\xi = -\infty$, where each term vanishes, gives
\begin{equation*}
    \frac{\d}{\d\xi}A(V) \;=\; B(V) - cV - \int_{-\infty}^{\xi} \tilde f\big(V(\eta)\big)\,\d\eta.
\end{equation*}
Every term is a function of $s = V(\xi)$ alone, which we may use in place of the variable $\xi$ when the profile is strictly increasing (which we prove in Theorem~\ref{thm:wave} below). Write $g(s)$ for the slope $V'$ at level $s$ (i.e.~the density of log capitalizations at the $s$-quantile) and
\begin{equation*}
    \varphi(s) \;=\; \frac{\d}{\d\xi}A(V) \;=\; \tfrac{1}{2}\,\sigma^2(s)\,g(s)
\end{equation*}
for the diffusive flux along the wave. Substituting $r = V(\eta)$ and $\d\eta = \d r/g(r)$ in the remaining integral leaves the \emph{wave identity}
\begin{equation}\label{eq:wave-identity-int}
    \varphi(s) \;=\; B(s) - cs - \int_0^s \frac{\tilde f(r)}{g(r)}\,\d r, \qquad s \in (0,1), \qquad \varphi(0) = \varphi(1) = 0.
\end{equation}
The boundary values hold because the profile flattens at its ends; parts (ii) and (iii) of Theorem~\ref{thm:wave} below make this precise. We note that the integrand $\tilde f/g$ is bounded: both $\tilde f$ and $g$ vanish at the endpoints, and the ratio tends to $\tilde f'(e)$ divided by the decay rate of $V'$ at the end state $e \in \{0,1\}$, again by Theorem~\ref{thm:wave}(ii).

Evaluating \eqref{eq:wave-identity-int} at $s = 1$, where the flux vanishes again, gives an identity for the speed of the wave:
\begin{equation}\label{eq:wave-speed}
    c \;=\; B(1) - \int_0^1 \frac{\tilde f(r)}{g(r)}\,\d r.
\end{equation}
Differentiating \eqref{eq:wave-identity-int} returns the wave identity in pointwise form,
\begin{equation}\label{eq:wave-identity-main}
    \varphi'(s) \;=\; \tilde b(s) - c - \frac{\tilde f(s)}{g(s)}.
\end{equation}
This identity ties the volatility, the drift, the turnover term, and the shape together pointwise in rank: prescribing any three determines the fourth. Subsection~\ref{subsec:drift-calib} uses the integrated identity \eqref{eq:wave-identity-int} to calibrate a rank-dependent drift, given the empirical volatility, reaction term, and capital distribution curve.

\begin{remark}\label{rem:quantile-ode}
Let $\Xi = V^{-1}$ be the quantile function of the wave, so that $\Xi(s)$ is the log capitalization at rank $s$ and $\Xi' = 1/g$.
Written in terms of $\Xi$, the wave identity \eqref{eq:wave-identity-int} becomes the first-order equation
\begin{equation}\label{eq:quantile-ode}
    \frac{\sigma^2(s)}{2\,\Xi'(s)} \;+\; \int_0^s \tilde f\,\d\Xi \;=\; B(s) - cs.
\end{equation}
Without turnover, $\tilde f = 0$ makes $c = B(1)$ by \eqref{eq:wave-speed}, and $\Xi'(s) = \sigma^2(s)/\big(2(B(s) - sB(1))\big)$ integrates to the equilibrium quantile function of the mean-field Atlas model in \cite[Section 2.3]{jourdain_capital_2015}. The condition from that work that the cumulative drift should lie above its chord is exactly the requirement that the flux $\varphi$ stay positive (see Subsection~\ref{subsec:beyond} for further comparison). With turnover, the equation is no longer explicit, since $\tilde f$ is integrated against $\Xi$ itself.
\end{remark}

\subsection{Existence, Uniqueness, and Tail Exponent}
\label{subsec:wave-exist}

We now prove that the wave exists and is unique for the full rank-dependent coefficient class. Our main reference is the monograph \cite{gilding_travelling_2004} on traveling waves in nonlinear diffusion-convection-reaction equations, whose theory is built on integral equations of the form \eqref{eq:wave-identity-int}.

\begin{theorem}\label{thm:wave}
Assume $\sigma^2, \tilde b \in C^1([0,1])$ with $\min_{[0,1]} \sigma^2 > 0$. Let $\tilde f \in C^1([0,1])$ be bistable: $\tilde f(0) = \tilde f(1) = 0$, and $\tilde f$ has exactly one interior zero $q^* \in (0,1)$, with $\tilde f < 0$ on $(0, q^*)$ and $\tilde f > 0$ on $(q^*, 1)$. Assume also that the three zeros are nondegenerate: $\tilde f'(0) < 0$, $\tilde f'(q^*) > 0$, and $\tilde f'(1) < 0$. Then:
\begin{enumerate}[label=(\roman*)]
\item There exist a speed $c \in \R$ and a strictly increasing profile $V \in C^2(\R)$ with $V(-\infty) = 0$ and $V(+\infty) = 1$, solving the traveling-wave equation \eqref{eq:wave-ode}. The pair is unique, in that no other speed admits a wavefront joining $0$ and $1$, and the profile is unique up to translation.
\item $V$ approaches each of its limits exponentially fast. The rate at the end state $e$ is a root of the characteristic equation
\begin{equation}\label{eq:char-poly}
    \tfrac{1}{2}\,\sigma^2(e)\, r^2 + \big(c - \tilde b(e)\big)\, r + \tilde f'(e) \;=\; 0, \qquad e \in \{0, 1\},
\end{equation}
the positive root at $e = 0$ and the negative root at $e = 1$.
\item The level parameterization of $V$ satisfies the wave identity \eqref{eq:wave-identity-main}, with boundary values $\varphi(0) = \varphi(1) = 0$.
\end{enumerate}
\end{theorem}

\begin{proof}
The reflection $y \mapsto -y$ turns our increasing profile into the decreasing one of \cite{gilding_travelling_2004} and reverses the sign of the speed; our flux $\varphi$ is their unknown $\theta$, and \eqref{eq:wave-identity-int} is, in our variables, their singular integral equation \cite[(1.9)]{gilding_travelling_2004}. With a bistable source and the convection term present, the relevant statement is \cite[Application 8.11]{gilding_travelling_2004}: there is exactly one speed $c$ for which a wavefront joining the two end states exists, and that wavefront is unique modulo translation. Its hypothesis $(ca')'(q^*) > 0$ is, in our notation, the transversality condition $\tilde f'(q^*) > 0$. Uniqueness of the speed and of the profile is also available separately as \cite[Theorems 8.3(vii) and 8.7(c)]{gilding_travelling_2004}, and that $V$ is $C^2$ with $V' > 0$ throughout is \cite[Theorem 2.39]{gilding_travelling_2004}. This proves (i).

The asymptotics of (ii) do not follow from \cite{gilding_travelling_2004}, whose endpoint results assume no convection, so we must argue this point using the phase plane, as in \cite[Section 4]{aronson_nonlinear_1975}. Linearizing \eqref{eq:wave-ode} at $e \in \{0,1\}$ produces the characteristic polynomial \eqref{eq:char-poly}, whose roots are real and of opposite sign because $\tilde f'(e) < 0$: both end states are saddles, the connection leaves $0$ along the positive root as $\xi \to -\infty$ and approaches $1$ along the negative root as $\xi \to +\infty$, and $V$, $V'$ and $V''$ converge exponentially at both ends. Finally, since $V$ is $C^2$ with $V' > 0$ and approaches its limits exponentially, every step of the derivation of \eqref{eq:wave-identity-int} applies to $V$ verbatim: the boundary terms at $-\infty$ vanish, $\tilde f(V)$ is integrable on left half-lines, and the level substitution is a $C^1$ change of variables. The boundary values $\varphi(0) = \varphi(1) = 0$ follow from the exponential decay of $V'$, which proves (iii).
\end{proof}

\begin{remark}
The transversality hypothesis $\tilde f'(q^*) > 0$, which the calibrated reaction term satisfies, is needed only for the rank-dependent existence theory. With constant coefficients the strict sign change alone suffices: existence and uniqueness of the front follow from \cite[Theorem 2.4(c) and Corollary 2.3]{fife_approach_1977}, with no transversality condition on $\tilde f$. In particular, transversality plays no part in the constant-coefficient relaxation theory of Subsection~\ref{subsec:relax-sub} (Theorem~\ref{thm:relaxation}).
\end{remark}

The rates in part (ii) hold in a logarithmic sense: for coefficients that are merely $C^1$, a two-sided comparison with the linearization at the saddle gives $\xi^{-1} \log|V(\xi) - e| \to r$, and this is all that we use in the sequel. When the derivatives of the coefficients are Lipschitz near the endpoints, as those of the calibrated model are, the linearization estimates sharpen to give $V - e \sim C \exp(r\xi)$; both statements are classical, see \cite[Chapter 13, Problems 5--7]{coddington_theory_1955}. The next corollary records the top rate as a Pareto tail exponent.

\begin{corollary}\label{cor:tail}
Under the hypotheses of Theorem~\ref{thm:wave}, the profile satisfies $1 - V(y) = e^{-\alpha y (1 + o(1))}$ as $y \to \infty$, where
\begin{equation}\label{eq:alpha-main}
    \alpha \;=\; \frac{a + \sqrt{a^{2} - 2\,\sigma^{2}(1)\,\tilde f'(1)}}{\sigma^{2}(1)},
    \qquad a \;=\; c - \tilde b(1):
\end{equation}
the capitalization distribution of the wave has a Pareto tail with exponent $\alpha$ at the top of the market. Symmetrically, $V(y) = e^{\beta y (1 + o(1))}$ as $y \to -\infty$, where $\beta$ is the positive root of \eqref{eq:char-poly} at $e = 0$.
\end{corollary}

\begin{proof}
By Theorem~\ref{thm:wave}(ii), $1 - V$ decays at the rate given by the negative root of \eqref{eq:char-poly} at $e = 1$, and $V$ grows at the positive root at $e = 0$; solving \eqref{eq:char-poly} for the former gives \eqref{eq:alpha-main}.
\end{proof}

An exponential tail in log capitalization is a Pareto tail in capitalization. Small $\alpha$ is concentration, and market diversity in the spirit of \cite[Definition 2.2.1]{fernholz_stochastic_2002} corresponds to $\alpha \geq 1$. In a finite sample drawn from the wave profile, the largest weight vanishes as the market grows when $\alpha > 1$, and remains a positive fraction when $\alpha < 1$. The finite-market version of this threshold is the phase transition of \cite{chatterjee_phase_2010}, whose Poisson--Dirichlet parameter plays the role of $\alpha$; in the mean-field setting, $\alpha$ is the (turnover-corrected) critical diversity index of \cite[Remark 2.5]{jourdain_capital_2015}, which we will return to quantitatively in Section~\ref{sec:portfolio}. We see that higher volatility at the top thickens the tail, while turnover at the top ($\tilde f'(1)$ more negative) thins it. Formally, setting $\tilde f'(1) = 0$ in \eqref{eq:alpha-main} with $a > 0$ collapses the formula to $\alpha = 2a/\sigma^2(1)$. This no-turnover value is the familiar drift-gap-to-variance ratio that sets capital distribution tails in Atlas-type models \cite{banner_atlas_2005, fernholz_stochastic_2002}, and turnover enters as the square-root correction. Subsection~\ref{subsec:drift-calib} evaluates \eqref{eq:alpha-main} on the calibrated wave. 

\begin{example}[Nagumo's Equation]\label{ex:nagumo}
If the calibrated $\tilde f$ is replaced by a cubic with the same zeros, $\kappa\,w(w - q^*)(1 - w)$ with $q^* \approx \FCross$, the constant-coefficient wave is logistic,
\begin{equation*}
    V(\xi) \;=\; \Big(1 + e^{-\sqrt{\kappa}\,\xi/\sigma}\Big)^{-1},
    \qquad
    c \;=\; \tilde\mu + \sigma\sqrt{\kappa}\,\Big(q^* - \frac{1}{2}\Big),
\end{equation*}
so the capitalization distribution is exactly log-logistic, with Pareto tails of exponent $\sqrt{\kappa}/\sigma$ at both ends. Our calibrated $\tilde f$ is more asymmetric than a cubic, so we solve for the wave numerically.
\end{example}

\subsection{Exponential Relaxation}
\label{subsec:relax-sub}

This subsection proves convergence to the wave in the constant-coefficient case, where \eqref{eq:pde-general} is semilinear. The classical bistable convergence theory, which we use directly, is due to \cite{fife_approach_1977}; the spectral framework for the stability of fronts is developed in \cite[Chapter 5]{henry_geometric_1981}. For the quasilinear case, \cite{meyries_quasi-linear_2014} is a user's guide to well-posedness, spectra, and stability of traveling waves. We use its semilinear specialization in Proposition~\ref{prop:rate} below, and Remark~\ref{rem:rank-dependent} points to the ingredients, assembled largely from that guide, by which the convergence rate is expected to extend to rank-dependent coefficients. The quantitative rate estimates behind Proposition~\ref{prop:rate} are from \cite{pruss_convergence_2009}.

\begin{theorem}[Exponential relaxation to the traveling wave]\label{thm:relaxation}
In the setting of Theorem~\ref{thm:wave}, assume in addition that $\sigma$ and $\tilde b \equiv \tilde\mu$ are constant. Then for every CDF initial condition $w_0$, there exist $y_0 \in \R$ and constants $K, \gamma > 0$ such that the unique bounded solution $w$ of \eqref{eq:pde-general} started from $w_0$ satisfies
\begin{equation}\label{eq:relaxation}
    \sup_{y \in \R}\, \big| w(t, y) - V(y - ct - y_0) \big| \;\leq\; K e^{-\gamma t}, \qquad t \geq 0,
\end{equation}
where $(c, V)$ is the wave of Theorem~\ref{thm:wave}.
\end{theorem}

\begin{proof} \textbf{Step 1: Reduction to Canonical Form.}
With $\sigma$ and $\tilde b \equiv \tilde\mu$ constant, \eqref{eq:pde-general} is the semilinear equation
\begin{equation}\label{eq:semilinear-form}
    \pp_t w \;=\; \frac{\sigma^2}{2}\,\pp^2_{yy} w - \tilde\mu\,\pp_y w + \tilde f(w),
\end{equation}
and the whole statement reduces to the classical bistable theory by a change of variables.
Passing to the frame $y \mapsto y + \tilde\mu t$ removes the drift: if $v(t,y) := w(t, y + \tilde\mu t)$ then $\pp_t v = \tfrac{1}{2}\sigma^2 \pp^2_{yy} v + \tilde f(v)$. Rescaling space by $z := y\sqrt{2}/\sigma$ removes the diffusion constant, and $U(t,z) := v(t, \sigma z/\sqrt 2)$ solves
\begin{equation}\label{eq:canonical}
    \pp_t U \;=\; \pp^2_{zz} U + \tilde f(U), \qquad z \in \R,\ t > 0.
\end{equation}
A traveling wave of \eqref{eq:canonical} with speed $c_0$ and profile $U_0$ corresponds to a traveling wave of \eqref{eq:semilinear-form} with
\begin{equation*}
    c \;=\; \tilde\mu + \frac{\sigma}{\sqrt 2}\,c_0,
    \qquad
    V(\xi) \;=\; U_0\big(\xi\sqrt{2}/\sigma\big),
\end{equation*}
and sup-norm distances between CDFs are unchanged by either substitution, so it suffices to prove the theorem for \eqref{eq:canonical}.

Extend $\tilde f$ off $[0,1]$ to a bounded Lipschitz function on $\R$. For bounded initial data, the Cauchy problem for \eqref{eq:canonical} then has a unique bounded solution, classical for $t > 0$: Duhamel iteration against the heat semigroup gives existence, and Gronwall's inequality gives uniqueness among bounded solutions. Since the constants $0$ and $1$ are solutions, the comparison principle keeps solutions with data in $[0,1]$ inside $[0,1]$; in particular the choice of extension does not matter. The constant-coefficient existence theory applies directly: \cite[Theorem 2.4(c) and Corollary 2.3]{fife_approach_1977} yield a strictly increasing front $U_0$ for \eqref{eq:canonical} with limits $0$ and $1$, unique together with its speed $c_0$ up to translation, whose endpoint rates in the canonical variables read $\tfrac{1}{2}\big[-c_0 \pm \sqrt{c_0^2 - 4\tilde f'(e)}\,\big]$ (the $+$ sign at $e = 0$, the $-$ sign at $e = 1$), consistent with the linearization in \cite[proof of Lemma 4.3]{fife_approach_1977}. Finally, for initial laws with a finite first moment these solutions are admissible in the sense of Lemma~\ref{lem:uniqueness}, the moment propagating by the truncation argument of the Step 2 moment bound in the proof of Theorem~\ref{thm:mean-field}; the waves are admissible as well, by their endpoint rates.

\paragraph{Step 2: Global Convergence.} This is \cite[Theorem 3.1]{fife_approach_1977}. That theorem assumes $\tilde f \in C^1[0,1]$ with $\tilde f(0) = \tilde f(1) = 0$, $\tilde f'(0) < 0$, $\tilde f'(1) < 0$, $\tilde f < 0$ on $(0,\alpha_0)$ and $\tilde f > 0$ on $(\alpha_1, 1)$ for some $0 < \alpha_0 \leq \alpha_1 < 1$, which our hypotheses give with $\alpha_0 = \alpha_1 = q^*$, together with the existence of a front joining the end states, which we get from Step 1; and it requires of the initial condition only that $0 \leq w_0 \leq 1$ with
\begin{equation}\label{eq:fm-data}
    \limsup_{y \to -\infty} w_0(y) \;<\; q^* \;<\; \liminf_{y \to +\infty} w_0(y).
\end{equation}
Its conclusion is that there are $y_0 \in \R$ and constants $K, \gamma > 0$ with $\sup_y |w(t,y) - V(y - ct - y_0)| \leq K e^{-\gamma t}$. A CDF satisfies \eqref{eq:fm-data} outright, its limits being $0$ and $1$ (the condition asks far less than being a CDF, but we do not pursue the wider class). One point of bookkeeping: the standing assumption of \cite{fife_approach_1977} is piecewise-continuous data, and a CDF may have jumps at a dense set of points. However, by Step 1 the solution is continuous for $t > 0$, stays in $[0,1]$, and retains the limits $0$ and $1$ at $y = \mp\infty$ (the reaction term vanishes at the end states, so the Duhamel iteration preserves the limits), so the theorem applies from any positive time. Enlarging $K$ over the initial interval, where $|w - V| \leq 1$, extends the bound to every $t \geq 0$. This proves \eqref{eq:relaxation}.
\end{proof}

Theorem 3.1 of \cite{fife_approach_1977} produces some rate $\gamma > 0$. The following proposition identifies how large the rate can be taken, in terms of the spectrum of the linearization about the front.

\begin{proposition}\label{prop:rate}
In the setting of Theorem~\ref{thm:relaxation}, with $\tilde f'$ locally Lipschitz, let $U_0$ and $c_0$ be the front and speed in the canonical frame \eqref{eq:canonical} from Step 1 of the proof, and let
\begin{equation*}
    \mathcal{L}\psi \;=\; \pp^2_{zz}\psi + c_0\,\pp_z\psi + \tilde f'(U_0)\,\psi
\end{equation*}
be the linearization about the front in the moving frame, acting on the space $C_{\mathrm{unif}}(\R)$ of bounded uniformly continuous functions. Then zero is a simple eigenvalue of $\mathcal{L}$, and the rate in \eqref{eq:relaxation} may be taken to be any $\gamma < \gamma_1$, where the spectral gap $\gamma_1 > 0$ is the distance from zero to the rest of the spectrum. Moreover,
\begin{equation}\label{eq:gamma-zero}
    \gamma_1 \;\leq\; \gamma_0 \;:=\; \min\big(|\tilde f'(0)|,\ |\tilde f'(1)|\big),
\end{equation}
where $\gamma_0$ is the edge of the essential spectrum of $\mathcal{L}$, and equality holds in \eqref{eq:gamma-zero} if and only if zero is the only eigenvalue of $\mathcal{L}$ above the essential spectrum.
\end{proposition}

\begin{proof}
Differentiating the wave ODE shows $\mathcal{L}U_0' = 0$, so zero is an eigenvalue with the positive eigenfunction $U_0'$. As $z \to \pm\infty$, the coefficients of $\mathcal{L}$ converge to those of the constant-coefficient operators $\mathcal{L}_e = \pp^2_{zz} + c_0 \pp_z + \tilde f'(e)$, $e \in \{0, 1\}$. The symbol of $\mathcal{L}_e$ at frequency $\zeta$ is $\tilde f'(e) - \zeta^2 + i c_0 \zeta$, whose real part is largest at $\zeta = 0$, where it equals $\tilde f'(e)$. For a scalar operator whose coefficients converge at $\pm\infty$, \cite[Theorem A.2 and the Example following it, Appendix to Chapter 5]{henry_geometric_1981} gives
\begin{equation}\label{eq:ess-spec}
    \sigma_{\mathrm{ess}}(\mathcal{L}) \;\subseteq\; \big\{\operatorname{Re}\lambda \leq \max(\tilde f'(0), \tilde f'(1))\big\},
\end{equation}
with equality at the edge. (In \cite{henry_geometric_1981}, the equation is written so that stability corresponds to spectrum in the right half-plane; we have translated accordingly.) Moreover, zero is a simple eigenvalue and the rest of the spectrum lies strictly to the left: Henry carries this out in \cite[\S5.4]{henry_geometric_1981} for the cubic instance of the front, and the argument applies verbatim to any bistable $C^1$ reaction term. By the definition of the essential spectrum in \cite[Appendix to Chapter 5]{henry_geometric_1981}, the spectrum to its right consists of isolated eigenvalues of finite multiplicity; these accumulate only in the essential spectrum, and $\mathcal{L}$ is sectorial \cite[Section 1.3]{henry_geometric_1981}, so only finitely many lie in any strip $\operatorname{Re}\lambda \geq -\gamma_0 + \epsilon$, and $\gamma_1 > 0$ is well defined. Since $\tilde f'(0)$ and $\tilde f'(1)$ are negative, the edge $\max(\tilde f'(0), \tilde f'(1))$ equals $-\gamma_0$, and so the essential spectrum reaches $-\gamma_0$. This proves \eqref{eq:gamma-zero}, and equality holds exactly when no eigenvalue of $\mathcal{L}$ lies in $(-\gamma_0, 0)$.

It remains to show that any rate $\gamma < \gamma_1$ can be attained in \eqref{eq:relaxation}. By Step 2 of the proof of Theorem~\ref{thm:relaxation}, the solution is eventually uniformly close to a translate of the front. We then apply the local convergence of \cite[Proposition 4.1]{meyries_quasi-linear_2014}; our equation is the semilinear case of that framework, and parabolic smoothing at any positive time upgrades sup-norm closeness to the closeness in H\"older norm required there (H\"older phase spaces are admitted by \cite[Theorem 2.6]{meyries_quasi-linear_2014}). That the rate can be taken to be any number below the spectral gap holds by \cite[(2.11) and (2.27)]{pruss_convergence_2009}. Our conclusion \eqref{eq:relaxation} is in the sup norm, which the H\"older norm dominates.
\end{proof}

\begin{remark}[Rank-dependent convergence]\label{rem:rank-dependent}
Theorems~\ref{thm:mean-field} and \ref{thm:wave} allow $\sigma$ and $b$ to depend on rank, in which case \eqref{eq:pde-general} is quasilinear and the reduction in Step 1 of the proof of Theorem~\ref{thm:relaxation} is no longer available. We expect that the convergence of Theorem~\ref{thm:relaxation} also holds in this setting, and with the same rate as in Proposition~\ref{prop:rate}: the essential spectrum of the quasilinear linearization is again determined by freezing the coefficients at the two ends \cite[\S3]{meyries_quasi-linear_2014}, and its edge depends only on the endpoint slopes of the reaction term, so \eqref{eq:gamma-zero} would not change. However, a detailed proof would require several technical ingredients which we do not pursue here: existence and interior regularity of solutions for arbitrary CDF initial data; a parabolic Harnack inequality for the contraction step (in the constant-coefficient case, \cite[Lemmata 4.1 and 4.4--4.5]{fife_approach_1977} trap the solution between two translates of the front and merge them with a Lyapunov argument; \cite[Lemma 3.3]{chen_existence_1997} provides an alternative in the quasilinear case, but uses a Harnack inequality as a hypothesis); and the verification of the hypotheses of the quasilinear maximal-regularity framework for our divergence-form equation (see \cite[Proposition 4.1]{meyries_quasi-linear_2014}, and \cite[\S5]{pruss_convergence_2009} where this is carried out for a quasilinear bistable front).
\end{remark}

\begin{remark}[Relaxation timescales]\label{rem:clock}
The bound $\gamma_0$ in Proposition~\ref{prop:rate} depends only on the endpoint slopes of the reaction term, and by the previous remark we expect the same for rank-dependent coefficients. We interpret $\gamma_0$ as the slowest relaxation timescale of the model. With the calibrated reaction term, the two endpoint timescales differ by two orders of magnitude: $\ln 2 / |\tilde f'(0)|$ is about \RelaxHalfBottomMonths{} months at the bottom of the market, $\ln 2 / |\tilde f'(1)|$ is about \RelaxHalfTopYears{} years at the top (bootstrap bands \BootHalfBotLoMonths{}--\BootHalfBotHiMonths{} months and \BootHalfTopLoYears{}--\BootHalfTopHiYears{} years; see Subsection~\ref{subsec:reaction}), and $\gamma_0$ is the slower of the two. This yields a \emph{two-timescale heuristic}: each end contributes one branch of the essential spectrum, so a deviation concentrated near one end is expected to decay at the rate of that end: within months at the bottom, and over decades at the leading edge. 
In practice, this asymmetry tells us where the data may be read as if in steady state. On decade horizons, the bulk of the distribution is relaxed while the leading edge is not; hence the calibrations of Section~\ref{sec:wave} are era-local statements about the bulk, and persistent deviations should be found at the leading edge (see Subsection~\ref{subsec:growth-acct}).
\end{remark}

\subsection{Comparisons with Known Cases}
\label{subsec:beyond}

We compare our bistable reaction-diffusion equation with two other known cases: the monostable equation corresponding to the 2000s, and the drift-stabilized models of \cite{jourdain_propagation_2013, jourdain_capital_2015,shkolnikov_large_2012}.

\paragraph{The Monostable Case.} 
As determined by the calibration in Subsection~\ref{subsec:reaction}, in the 2000s, exits outpaced entries and $\tilde f$ fell below zero at essentially every rank; we idealize the decade as monostable, with $\tilde f < 0$ on $(0,1)$, $\tilde f'(0) < 0$, and $\tilde f'(1) > 0$. 
The derivation of the wave identity \eqref{eq:wave-identity-int} uses only a strictly increasing $C^2$ profile with limits $0$ and $1$ and decaying ends; it never uses the sign of $\tilde f$. Hence the shape of the capital distribution curve may still be inferred from \eqref{eq:wave-identity-int}, and for each admissible speed the profile is still unique up to translation \cite[Theorem 8.7(b)]{gilding_travelling_2004}. With $\tilde f'(1) > 0$, the top end state is a node rather than a saddle, and the equation no longer admits a unique speed. The saddle-to-saddle connection of Theorem~\ref{thm:wave} fixes the speed because it is a codimension-one coincidence in the phase plane; a saddle-to-node connection instead survives perturbations of $c$. In the latter case, wavefronts exist for a closed half-line of speeds $[c^*, \infty)$ rather than for a single speed (see \cite[Theorems 8.2 and 8.3(ii)]{gilding_travelling_2004}; nonemptiness of the speed set reduces to a one-line test-function check). Therefore, the coefficients no longer determine the speed or the tail, and initial data with heavier tails are expected to select faster fronts with heavier tails.
Stability degenerates as well: with $\tilde f'(1) > 0$ the essential spectrum of the linearization enters the right half-plane (the equality at the edge in \eqref{eq:ess-spec} does not use the signs of the slopes), so no single front attracts all CDF data uniformly, and deviations at the top of the market are unresolved for a monostable reaction term.
Nevertheless, one observation holds for the whole family of monostable fronts. With $\tilde f \leq 0$, the reaction term lowers the normalized CDF at every level, so every quantile of the surviving population climbs, and the median can rise while the typical firm at every rank stagnates. This is survivorship rather than growth. The speed identity makes it quantitative: every admissible front has $c \geq B(1)$ and outruns the average firm. The 2000s entry in Figure~\ref{fig:growth-decomp} below shows exactly this.

\paragraph{Drift-Stabilized Models.}

Setting $\tilde f = 0$ in \eqref{eq:pde-general} yields the Cauchy problem of \cite{jourdain_propagation_2013, jourdain_capital_2015,shkolnikov_large_2012}. 
As discussed in Remark~\ref{rem:atlas}, we do not pursue the same degeneracy of $\sigma$, as we assume uniform ellipticity. In exchange, the wave theory of Theorem~\ref{thm:wave} covers the full rank-dependent class, the convergence in Theorem~\ref{thm:mean-field} is uniform, and Proposition~\ref{prop:rate} gives a rate; we use all three in Sections~\ref{sec:wave} and \ref{sec:diversity} below.
The law of large numbers of \cite{shkolnikov_large_2012} starts the system at stationarity: the initial spacings are drawn from their stationary law, which the strictly decreasing drift provides through the theory of reflected Brownian motion recalled in Subsection~\ref{subsec:background}.
Equilibration in \cite{jourdain_propagation_2013, jourdain_capital_2015} requires the chord condition: the equilibrium exists when $\Gamma(u) := \int_0^u \tilde b$ satisfies $\Gamma(u) > u\,\Gamma(1)$ on $(0,1)$, so that the cumulative drift lies above its chord, the continuum form of the stability condition for the finite Atlas model \cite{banner_atlas_2005}. In the setting of \eqref{eq:wave-identity-int}, the chord condition yields positivity of the flux $\varphi$, and their market growth rate $\Gamma(1)$ is the speed \eqref{eq:wave-speed} evaluated at $\tilde f = 0$; these are the entropy condition and the Rankine--Hugoniot shock speed of the underlying conservation law (for background, see \cite[Section 3.4]{evans_partial_2010}). With turnover, the speed deviates from $\Gamma(1)$ by the turnover term of \eqref{eq:wave-speed}.

Our long-term behaviour comes from the bistable reaction term rather than the drift. We read this as a third stabilization mechanism for the capital distribution, alongside the drift-stabilized models and the volatility-stabilized markets of \cite{fernholz_relative_2005, pal_analysis_2011}, in which the volatility of small firms grows large. 
In contrast with \cite{jourdain_propagation_2013, shkolnikov_large_2012}, we require neither monotonicity of the drift, nor affine volatility, nor the chord condition; we do require uniform ellipticity, which their existence theory does not \cite[Proposition 4.1]{jourdain_propagation_2013}. Turnover also does away with their remaining hypothesis: $\tilde f'(e) < 0$ makes the roots of \eqref{eq:char-poly} real and of opposite sign, so the flux vanishes linearly at both ends regardless of the drift, which without turnover must be imposed as the integrability condition (E2) of \cite{jourdain_capital_2015}. Subsection~\ref{subsec:drift-calib} below recovers a drift curve from the wave identity and finds the chord condition violated at every rank, by up to \DriftChordDevMax{} per year; the naive alternative, the drift measured from surviving firms, violates it at \EOneFailShareMeas\% of ranks. With the empirically identified drift and no turnover, the drift-stabilized equilibrium does not exist at all (see also Table~\ref{tab:wave-variants}). 

The convergence comparison with \cite{jourdain_capital_2015} is two-sided. Both convergence theorems assume uniform ellipticity; theirs covers rank-dependent coefficients, while Theorem~\ref{thm:relaxation} is proved for constant coefficients, with Remark~\ref{rem:rank-dependent} stating our expectation for the rank-dependent case. For the semilinear case, we obtain convergence in the uniform norm with a rate. For $\tilde f = 0$, the mean is conserved in the moving frame: the conservation fixes the translate of the limit in \cite[Theorem 2.4]{jourdain_capital_2015}, and it leaves the essential spectrum of the linearization touching the origin, so the convergence proved in \cite{jourdain_propagation_2013} is in the Wasserstein distance and carries no rate in general (exponential rates are recovered there for constant diffusion and initial data close to equilibrium). Their flow is nevertheless a contraction: the Wasserstein distance between two solutions is nonincreasing, even with degenerate diffusivity \cite[Proposition 3.1]{jourdain_propagation_2013}. Since $W_1$ in one dimension is the $L^1$ distance between CDFs, this is $L^1$ contraction of the conservation-law semigroup, and it gives stability without a spectral gap. The reaction term changes both conclusions: it breaks the conservation of the mean, so the shift $y_0$ in Theorem~\ref{thm:relaxation} is determined by the initial data rather than by a conservation law, and it opens the gap \eqref{eq:gamma-zero}, which Theorem~\ref{thm:relaxation} and Proposition~\ref{prop:rate} turn into an exponential rate in the supremum norm, and which Subsection~\ref{subsec:growth-acct} tests as a timescale.

\section{The Capital Distribution as a Traveling Wave}
\label{sec:wave}

This section examines the capital distribution curve through the wave identity \eqref{eq:wave-identity-int}, which can be used either for prediction or calibration. Subsection~\ref{subsec:ranked-vol} compares the empirical curve with a predicted curve. We measure the volatility curve rank by rank, fit the intensities from event counts, and hold the drift constant; the shape of the wave is then fully determined, since a constant drift enters the traveling-wave equation only through the speed of the front. We compare the measured wave with the constant-volatility baseline, and with a variant that instead uses the naive rank-dependent drift measured from surviving firms; the two carry errors of the same size. Subsection~\ref{subsec:drift-calib} uses \eqref{eq:wave-identity-int} for calibration: we recover a rank-dependent drift using the observed shape, the volatility curve, and the turnover term. The recovered drift tightly calibrates the capital distribution curve together with its Pareto tail. Subsection~\ref{subsec:growth-acct} then tests the wave decomposition against the century of data: the growth of the market decomposes into translation of the wave plus entry, and the persistence of shape deviations follows the predicted timescales.

\subsection{The Predicted Wave}
\label{subsec:ranked-vol}

The calibrated reaction term of Subsection~\ref{subsec:reaction} is bistable, so by Theorem~\ref{thm:wave}, the long-run behaviour of our model is given by a unique wave. We identify the wave with the long-run capital distribution curve. The attraction of the wave is proved in the constant-coefficient case (Theorem~\ref{thm:relaxation}) and exhibited numerically in the rank-dependent case below; and in either case, the bound on the rate involves only the endpoint slopes of the reaction term (Proposition~\ref{prop:rate} and Remark~\ref{rem:rank-dependent}). The resulting economic interpretation is simple. In the long run, the capital distribution has a \emph{fixed shape} that translates along the log-capitalization axis at a constant speed: the market grows, but the profile of the distribution (how far the median firm is from the largest firms, how thick the lower tail is) does not change. This matches the empirical stability of the capital distribution documented by Fernholz \cite{fernholz_stochastic_2002}. 

\paragraph{Measurement Conventions.} We state our conventions for the measured inputs used in this section and the next. The scalar coefficients of the constant baseline are cross-sectional medians over firms over the window 1975--2024: $\tilde\mu$ from the mean monthly changes in log capitalization, and $\sigma$ from the realized volatilities of monthly log returns. Throughout, drifts are capitalization growth and volatilities are return volatilities; share issuance is lumpy and we put it with the drift rather than with the diffusion. We use an implicit Euler scheme with the coefficients lagged by one time step: a step of $0.1$ months, a uniform grid of $2000$ points spanning the January 1975 cross-section with margin, boundary values $0$ and $1$, and the January 1975 empirical CDF as initial condition. Profiles become capital distribution curves by reading quantiles at the final CRSP firm count.

The rank-dependent volatility curve is estimated from monthly log returns over 1975--2024: fifty equal rank bins, the standard deviation within each bin, and a natural cubic spline smoothing of the binned values, annualized by $\sqrt{12}$. The measured curve is used on ranks $s \in [0.02, 0.98]$ and extended constantly beyond; in particular, $\sigma^2(1)$ in the tail formula \eqref{eq:alpha-main} denotes the value at $s = 0.98$. Tail exponents are least-squares slopes of log weight on log rank over ranks $3$ through $100$, applied identically to data and model curves; the Hill estimator (a standard estimator of tail exponents, see \cite{hill_simple_1975}) on the same window gives \AlphaDataHill{} for the data and \AlphaRecWaveHill{} for the calibrated wave, agreeing with the slopes below. The quantile density $g$ is a Gaussian kernel estimate of the log capitalizations pooled over the final twenty-four months of the sample, each month centered at its cross-sectional median, smoothed by the same spline family. The smoothed curves are $C^1$ with $\min \sigma^2 > 0$, as Theorems~\ref{thm:mean-field} and \ref{thm:wave} require of the coefficients. In the decade panels below, the intensities and the volatility are re-estimated on each decade from twenty-five-bin histograms, the model curve is read at the decade's median firm count, and the target is the geometric mean of the decade's monthly capital distribution curves, truncated to the ranks covered in every month.

\paragraph{Constant Coefficients.}
With constant coefficients, the limiting equation for the normalized CDF $w$ in log coordinates is
\begin{equation}\label{eq:wave-pde}
    \pp_t w \;=\; \frac{\sigma^2}{2}\,\pp^2_{yy} w - \tilde\mu\,\pp_y w + \tilde f(w).
\end{equation}
Theorems~\ref{thm:wave} and \ref{thm:relaxation} apply. The wave speed decomposes into two parts: the growth rate $\tilde\mu$ of a typical firm, plus a turnover correction $c - \tilde\mu$ generated by the asymmetry of entries and exits.

Per the measurement conventions, the two scalar coefficients are $\tilde\mu = \MuTildeAnnual$ and $\sigma = \SigmaAnnual$ (annualized). Running the calibrated equation \eqref{eq:wave-pde} to long horizons settles the profile into its traveling shape (profile drift in Table~\ref{tab:wave-variants}). The front speed is $c = \WaveSpeedAnnual$ per year in log capitalization, against the measured typical-firm growth rate $\tilde\mu$. The difference $c - \tilde\mu = \TurnoverCorrAnnual$ per year is the turnover correction: the amount by which entry and exit advance the median of the distribution, over and above the growth of the firms themselves. However, the shape fails to match the data in the tails. With a single volatility parameter, the wave concentrates more weight in the top few firms, and less in the middle, than the data; the empirical capital distribution curve is close to a straight line over roughly 1000 ranks.

The constant-coefficient curve appears in Figure~\ref{fig:zerocalib} below, which compares the resulting capital distribution curve with the CRSP cross-section at the end of the sample.
We use the standard yardstick of stochastic portfolio theory: we plot the market weight of the $k$-th largest firm against its rank $k$, both on logarithmic scales (see \cite{fernholz_stochastic_2002}). By custom, the curve indexes firms with $k = 1$ as largest, while our model's rank quantile variable runs the other way, with $q = 1$ the largest. Each wave profile is converted into such a curve by reading off its quantiles at the same number of firms as in the final CRSP cross-section. Table~\ref{tab:wave-variants} collects the convergence and fit diagnostics for every variant, including the calibrated model of the next subsection. Following the practice of \cite{itkin_calibrated_2026}, we measure fit by the root-mean-square relative weight error over the top ranks.

\paragraph{Rank-Dependent Volatility.}
Next, we look at the shape of the wave with constant drift and rank-dependent volatility. The volatility curve is directly measurable: rank-binned monthly log returns give an essentially monotone curve, falling from \MeasVolBottom{} annualized for the smallest firms to \MeasVolTop{} for the largest. As discussed for a single Brownian motion in Subsection~\ref{subsec:background} (and see Remark~\ref{rem:drift-estimation} below for rank-based models), the drift is the ingredient the data pins down poorly. However, a constant drift enters the traveling-wave equation only through the speed of the front, since $\tilde b \equiv \tilde\mu$ appears in \eqref{eq:wave-ode} only through the difference $c - \tilde\mu$. Therefore, fixing the drift to any constant determines the shape of the wave using just the measured volatility curve and the fitted intensities: we call this the \emph{zero-calibration wave}.

We solve the quasilinear equation with these coefficients by an implicit scheme, with the coefficients lagged one time step, and run it to $t = 1200$ months. The solution settles into a traveling profile: recentered profiles at $t = 900$ and $t = 1200$ agree to within \ZCConv{} in sup norm (compare \CSConv{} in the constant-coefficient case). The resulting curve and diagnostics are presented in Figure~\ref{fig:zerocalib} and Table~\ref{tab:wave-variants}, respectively. The zero-calibration wave attains \ZCILHundred{} over the top 100 firms, against \CSILHundred{} for the constant-volatility baseline. The naive alternative to a constant drift is the measured rank-dependent drift: the rank-binned average month-over-month change in log capitalization of surviving firms, demeaned, since only its variation across ranks affects the shape. Substituting it for the constant gives \BMILHundred{} over the top 100 firms, an error of the same size. The survivorship bias of the naive measurement is still visible, concentrated at the bottom ranks (Figure~\ref{fig:drift-recovery}), and in Subsection~\ref{subsec:drift-calib} we invert \eqref{eq:wave-identity-int} to extract the rank structure that both models miss.

\begin{figure}[p]
\centering
\includegraphics[width=0.85\textwidth]{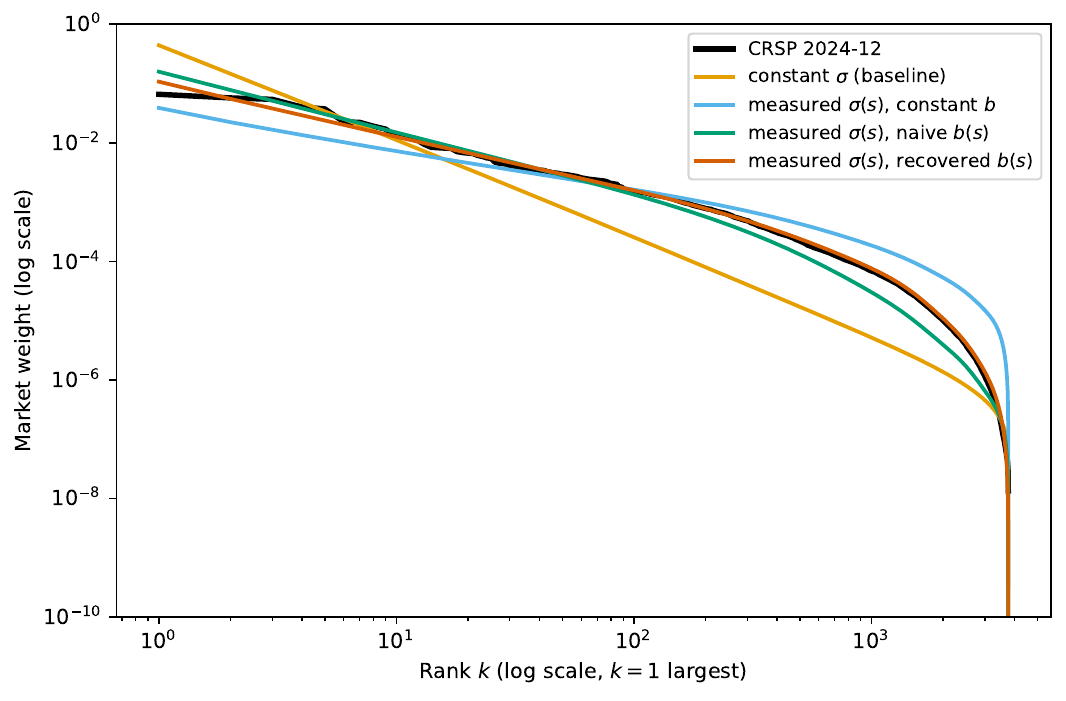}
\caption{Capital distribution curves of the traveling waves against the CRSP cross-section at the end of the sample, in order of increasing input: the constant-volatility baseline; the constant-drift \emph{zero-calibration }wave; the variant using the naive rank-dependent drift measured from surviving firms; and the wave with drift recovered from \eqref{eq:wave-identity-int}.}
\label{fig:zerocalib}
\end{figure}

\begin{table}[p]
\centering
\caption{Diagnostics for the traveling-wave variants. Profile drift is the sup-norm change of the recentered profile between $t = 900$ and $t = 1200$ months. Fit is the root-mean-square relative market-weight error against the CRSP cross-section; entries in parentheses are snapshots of the non-converging no-turnover solutions at $t = 1200$ and keep changing. The interquartile width in the data is \WidthData{} log units. The lower panel repeats the test decade by decade, with intensities and volatility re-estimated on each decade alone, under the constant drift zero-calibration wave and under the drift calibrated from each decade's own capital distribution.}
\label{tab:wave-variants}
\begin{tabular}{lcccc}
\toprule
Coefficients & Profile drift & \multicolumn{2}{c}{RMS relative weight error} & IQR width \\
\cmidrule(lr){3-4}
 & ($t = 900$ vs $1200$) & top 100 & top 1000 & at 100 years \\
\midrule
constant $\sigma$, constant $b$ & 0.0008 & 0.89 & 0.91 & 2.3 \\
measured $\sigma(s)$, constant $b$ & 0.0009 & 0.25 & 1.04 & 2.4 \\
measured $\sigma(s)$, naive $b(s)$ & 0.0042 & 0.19 & 0.41 & 3.6 \\
measured $\sigma(s)$, recovered $b(s)$ & 0.0020 & 0.11 & 0.09 & 3.7 \\
no turnover, measured $\sigma(s)$ & 0.0368 & (0.24) & (0.56) & 7.9 \\
no turnover, measured $\sigma(s)$, $b(s)$ & 0.0694 & (0.33) & (0.77) & 14.0 \\
\bottomrule
\end{tabular}
\\[8pt]
\begin{tabular}{lcccccccccc}
\toprule
Top-100 error & 1930s & 1940s & 1950s & 1960s & 1970s & 1980s & 1990s & 2000s & 2010s & 2020s \\
\midrule
constant $b$ & 0.50 & 0.28 & 0.15 & 0.20 & 0.27 & 0.27 & 0.60 & 0.32 & 0.31 & 0.31 \\
recovered $b(s)$ & 0.05 & 0.12 & 0.10 & 0.16 & 0.09 & 0.81 & 0.60 & 0.33 & 0.27 & 0.40 \\
\bottomrule
\end{tabular}

\end{table}

\paragraph{Comparisons.}
In Table~\ref{tab:wave-variants} we report two further checks. For the first, we remove turnover by setting $\lambda_b = \lambda_d = 0$. The solvers then never settle, whether the coefficients are constant, the volatility alone is rank-dependent, or the volatility and drift both are. In the constant case, the width grows like $\sqrt{\sigma^2 t}$ and the capital distribution curve flattens without limit. In the rank-dependent cases, the recentered profiles are still drifting an order of magnitude faster at $t = 1200$, and the interquartile width grows without bound. We conclude that the measured drift curve is too flat to stabilize the market on its own: with coefficients taken from the data, it is turnover that stabilizes the shape. For the second check, we repeat the comparison decade by decade (lower panel of the table). We refit the intensities and remeasure the volatility curve on each decade separately. The zero-calibration wave converges in every one of the \NDecAll{} decades (in the monostable 2000s the settled front is the one selected by the decade's initial data, see Subsection~\ref{subsec:beyond}), and its top-100 relative error against that decade's average capital distribution curve ranges from \DecErrMin{} to \DecErrMax{}, with a median of \DecErrMed. Hence a model with no calibrated shape parameters, refit on decade-length windows, stays within tens of percent of each era's capital distribution.

\subsection{Calibrating the Drift}
\label{subsec:drift-calib}

The wave identity \eqref{eq:wave-identity-int} links together the shape, the volatility, the reaction term, and the drift. These four quantities are not on the same statistical footing: the first three can be estimated at fast rates, while the drift cannot. The quantile density $g(s)$ is estimated from a short pool of recent cross-sections; $\sigma^2(s)$ from quadratic variation, with precision set by the sampling frequency; and $\tilde f$ from event counts. On the other hand, the drift can only be learned from the length of the observation window, as discussed for a single Brownian motion in Subsection~\ref{subsec:background}. A time-series estimate of a rank-binned drift carries a standard error of $\sigma(s)/\sqrt{T}$, which with volatilities up to \MeasVolBottom{} annualized is of the same order as the quantity being estimated, even over the full century. Moreover, an estimate from the growth of surviving firms conditions on survival, which biases the measurement exactly where exits concentrate (as seen in Subsection~\ref{subsec:ranked-vol}). Therefore, we treat the drift as the unknown in the wave identity, and we calibrate it from the other three quantities. Solving \eqref{eq:wave-identity-int} for the cumulative relative drift gives
\begin{equation}\label{eq:drift-integrated}
    \int_0^s \big(\tilde b(r) - c\big)\,\d r \;=\; B(s) - cs \;=\; \varphi(s) + \int_0^s \frac{\tilde f(r)}{g(r)}\,\d r.
\end{equation}
The front speed $c$ enters only as an additive constant. The identity therefore delivers the rank profile of the drift, relative to the speed of the front, from cross-sectional and quadratic-variation measurements alone. The one number which still requires a long time span is the speed $c$ itself. The speed identity \eqref{eq:wave-speed} does not supply it: $B(1)$ is itself a drift, and the identity is invariant under adding a common constant to $\tilde b$ and to $c$, so it determines only the turnover correction $c - B(1)$.

\begin{remark}[Drift estimation in rank-based models]\label{rem:drift-estimation}
In \cite[Sections 5.3--5.4]{fernholz_stochastic_2002}, the growth rates are deduced from the long-term behaviour of the collision local times, assuming convergence to stationarity. \cite{campbell_macroscopic_2025} construct a discrete-time analogue of these local times that accounts for firms switching ranks between consecutive time steps. Both constructions read the drift off the time axis: they require the ranked system to be near its stationary regime, and their precision grows with the length of the observation window. The pointwise wave identity \eqref{eq:wave-identity-main} is the mean-field, cross-sectional counterpart of the same idea: in the limit, the collision local time accumulated per unit time at level $s$ should converge to the diffusive flux $\varphi(s) = \tfrac{1}{2}\sigma^2(s)\,g(s)$, by the occupation-time formula. The identity then recovers the drift from that flux together with the shape and the turnover term, from a single cross-section rather than a long path of the ranked system (in practice, we pool a short window of months to stabilize the density estimate).
\end{remark}

The wave identity is available both pointwise \eqref{eq:wave-identity-main} and in integrated form \eqref{eq:drift-integrated}; we estimate the integrated form, to avoid differentiating an estimated curve, and extract the drift curve as the slope of a fit to $B$ in the same spline family as the measurement conventions. Figure~\ref{fig:drift-recovery} shows the result of the calibration. On the left, $B(s) - cs$ lies below its chord, with maximal deviation \DriftChordDevMax{} per year. Its slope rises from $\DriftRelBottom$ per year at $s = 0.12$ to $\DriftRelMid$ at the median rank and $\DriftRelTop$ at $s = 0.95$, so every rank drifts below the front speed. On the right, the raw pointwise identity carries ripples from differentiation, which the integrated form and its smoothed slope avoid. The naive survivor curve of Subsection~\ref{subsec:ranked-vol} is near the recovered drift except at the bottom ranks, where conditioning on survival hides the shrinkage of exit-bound firms. It is positioned by subtracting the empirical front speed, the trend \CEmpAnnual{} per year of the monthly median log capitalization over the same window; for comparison, the constant-coefficient wave of Subsection~\ref{subsec:ranked-vol} has speed \WaveSpeedAnnual{} per year.

\begin{figure}[t]
\centering
\includegraphics[width=\textwidth]{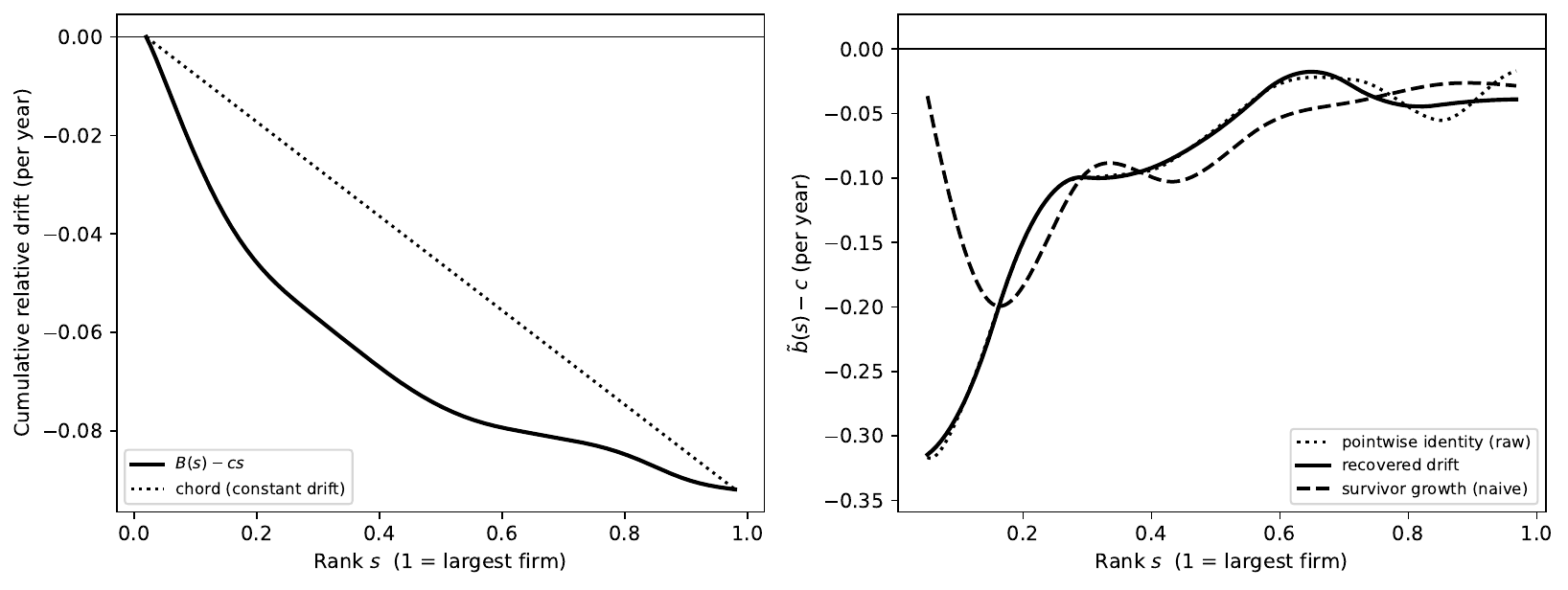}
\caption{Calibrating the drift from the wave identity. Left: the integrated form $B(s) - cs$ of \eqref{eq:drift-integrated} against the linear benchmark of a rank-independent drift. Right: the pointwise identity \eqref{eq:wave-identity-main} computed directly (dotted), the recovered drift from the slope of a spline fit to $B$ (solid), and the rank-binned capitalization growth of surviving firms (dashed). The volatility curve and survivor drift are measured over 1975--2024, and the quantile density over the final twenty-four months, as fixed in Subsection~\ref{subsec:ranked-vol}; the survivor curve is centered by the empirical front speed \CEmpAnnual{} per year.}
\label{fig:drift-recovery}
\end{figure}

As a final consistency check, we feed the recovered curve back into the solver. The resulting wave matches the CRSP capital distribution curve with relative error \RECILHundred{} over the top 100 firms and \RECILThousand{} over the top 1000 (see Figure~\ref{fig:zerocalib} and Table~\ref{tab:wave-variants}). Its tail nearly reproduces the empirical Pareto exponent of \eqref{eq:alpha-main}: over the top 100 ranks the calibrated wave gives $\alpha = \AlphaRecWave$, against \AlphaData{} in the data, where the zero-calibration wave gives \AlphaZCWave{}. The decade panel of Table~\ref{tab:wave-variants} repeats the test with each decade's own calibrated drift: the median top-100 error falls from \DecErrMed{} to \DecErrRecMed, with the largest gains in the early decades (\DecErrRecMin{} in the 1930s) and one clear miss in the 1980s.

The recovery inherits the assumption that the market is on its wave. The bulk of the recovered curve is insensitive to this assumption, since on the two-timescale reading of Remark~\ref{rem:clock} the bulk relaxes within months. The leading edge is not: a persistent deviation of the top ranks is absorbed into the recovered drift there, so the calibration is era-local at the top (Subsection~\ref{subsec:relax-sub}). The next subsection measures these deviations directly.

\subsection{The Wave and the Evolving Market}
\label{subsec:growth-acct}

We now test the wave decomposition with the century of data, through two model-agnostic measurements. The first is an exact growth accounting: total capitalization grows by translation of the front, by the entry of firms, and by deformation of the shape. If the market is moving as a traveling wave then the deformation term vanishes identically. Decade by decade the prediction largely holds, except in a handful of concentration episodes, where the shape term dominates. The relaxation timescales of Remark~\ref{rem:clock} make sense of the exceptions: the bulk of the market relaxes within months and cannot sustain a deviation, while the leading edge takes decades. Hence the episodes must live at the top, and the one era clearly apart from its wave, the 2020s, deviates at the top. The second measurement of this subsection tests the timescales themselves: the persistence of shape deviations, estimated rank by rank, rises from months at the bottom to decades in the top percent, in the order and on the scale the fitted intensities predict.

The wave speed $c$ is the growth rate of the \emph{median capitalization}. At any time,
\begin{equation*}
    \text{total cap} \;=\; n_t \,\times\, \text{median cap} \,\times\, \frac{\text{mean cap}}{\text{median cap}},
\end{equation*}
and taking logarithms turns the product into a sum whose derivative can be read term by term:
\begin{equation}\label{eq:growth-identity}
    \frac{\d}{\d t}\log(\text{total cap})
    \;=\; \underbrace{\frac{\d}{\d t}\log(\text{median cap})}_{\text{wave speed}}
    \;+\; \underbrace{\frac{\d}{\d t}\log n_t}_{\text{firm count}}
    \;+\; \underbrace{\frac{\d}{\d t}\log \frac{\text{mean cap}}{\text{median cap}}}_{\text{shape deformation}}.
\end{equation}
In the steady-wave regime, where the market moves along the wave of Theorem~\ref{thm:wave}, the third term vanishes: firms are added at rate $f(1)$, the wave translates at speed $c$, and the shape does not move. The identity itself involves no model coefficients. Volatility, whether constant or rank-dependent, enters only through the interpretation of the median term as the speed of a wave. Our benchmark for the shape is the zero-calibration wave of Subsection~\ref{subsec:ranked-vol}. The calibrated wave would be circular here: its drift is recovered from the observed shape, so deviations from it cannot measure deformation. Over the full century the wave picture is a good approximation: total capitalization grew at \GrowthTotalFull\% per year, of which \GrowthMedianFull\% is translation of the wave, \GrowthCountFull\% is growth in the number of firms, and only \GrowthShapeFull\% is shape deformation (the count term is a raw count, and the coverage expansions inside it are separated from economic entry below).

Figure~\ref{fig:growth-decomp} carries out the decomposition decade by decade. We note that the tall count bars of the 1960s and 1970s include the CRSP universe expansions of 1962 and 1972, which add firms without economic entry: the cleaned net entry rate of Subsection~\ref{subsec:data} is \NetGrowthAnnualPct\% per year, which separates the two. In most decades, the shape term is small and the market grows by translation and entry. The exceptions are the concentration episodes. Over the final ten years of the sample, 2015--2024, total capitalization grew at \GrowthTotalRecent\% per year while the median capitalization \emph{shrank} ($\GrowthMedianRecent\%$ per year); the shape term contributed \GrowthShapeRecent\%. On this reading, the market of the 2020s is off its wave at the leading edge. Provided the underlying coefficients remain unchanged (which does not hold decade to decade, as seen in Subsection \ref{subsec:reaction}; and see also the regime discussion in Section~\ref{sec:discussion}), the model does not suggest a quick return: leading-edge deviations decay on the top timescale, whose half-life is measured in decades, as we will see below.

\begin{figure}[htbp]
\centering
\includegraphics[width=0.86\textwidth]{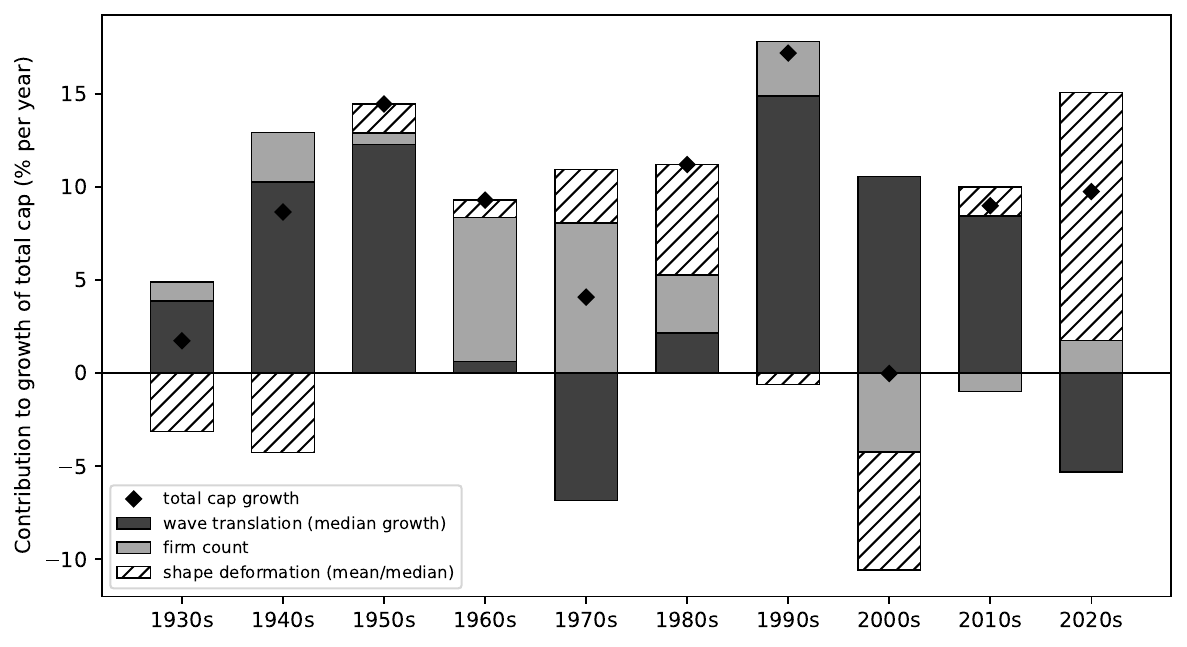}
\caption{The growth identity \eqref{eq:growth-identity} decade by decade: each bar decomposes the decade's growth of total market capitalization into translation of the wave (growth of the median capitalization), growth in the number of firms, and deformation of the shape; diamonds mark the total. Decades are calendar decades, with the 2020s bar covering 2020--2024. The 2020s are the only decade in which shape deformation exceeds the decade's total growth.}
\label{fig:growth-decomp}
\end{figure}

We recall which statements about the timescales are proved and which are heuristic. The proved statement concerns the global rate: it is bounded by $\gamma_0$, the edge of the essential spectrum, and with the calibrated reaction term, $\gamma_0$ is set by the top slope $\tilde f'(1)$ with half-life \RelaxHalfTopYears{} years. The two-timescale heuristic of Remark~\ref{rem:clock} goes further and splits the market: reading the two branches of the essential spectrum separately assigns the bottom of the market a timescale of \RelaxHalfBottomMonths{} months, from $\tilde f'(0)$, and the top a timescale of \RelaxHalfTopYears{} years, from $\tilde f'(1)$. Both involve only the endpoint slopes of the reaction term, so we expect them to be insensitive to the rank dependence of the diffusion and drift (Remark~\ref{rem:rank-dependent}). On decade horizons, then, the bulk of the distribution relaxes (evidenced by Table~\ref{tab:wave-variants}) while the leading edge is unresolved (concentration episodes above).
As a final test, we measure the persistence of shape deviations directly: centered quantiles decorrelate with a half-life of about \RelaxEmpBottomDemeanMonths{} months at the bottom of the market, and of decades in the top percent (\RelaxEmpTopYears{} years, against \RelaxHalfTopYears{} calibrated), in the order and on the scale predicted. We caveat that a century of data contains barely two half-lives of the slow mode.\footnote{Half-lives are exponential fits to autocorrelations of the monthly $s$-quantile of log capitalization minus its cross-sectional median; the bottom band is the $s = 0.05$ quantile and the top the $s = 0.99$ quantile. The bottom estimate is decade-demeaned on the post-1973 universe, so that slow drift in the intensity levels does not show up as a slow timescale (the filter is transparent at the bottom's predicted months scale). Each band also carries the other filter's reading: the raw bottom series has a \RelaxEmpBottomYears-year component, which we attribute to those moving levels (Subsection~\ref{subsec:reaction}), and decade-demeaning the top collapses it to \RelaxEmpTopDemeanMonths{} months by construction, since the filter removes anything slower than a decade; hence the top is read raw. The slow fits extrapolate beyond the 180-month lag window, and the corroboration is not an independent test: the persistence is measured on the century of data, which contains the half-century that fitted the intensities.}

\section{Growth of Diversity-Weighted Portfolios}
\label{sec:diversity}\label{sec:portfolio}

In this section, we use the wave to determine the capitalization growth of simple portfolio rules from stochastic portfolio theory in closed form. For the market moving along the wave, Proposition~\ref{prop:portfolio} gives the growth rate of any diversity-weighted portfolio below the critical diversity index as an explicit integral against the profile. The wave identity eliminates the drift from the integral, and turnover enters as a drag whose size is set by the fitted intensities. We evaluate the formula with the measured inputs for the classical rules, and test it decade by decade against the realized growth of the market portfolio. Diversity itself becomes an estimated functional of the model rather than an assumption. The fifty-year calibrated market is diverse by a thin margin, and the margin vanishes on the coefficients of the most recent fifteen years.

\subsection{Diversity-Weighted Portfolios}

For $p \geq 0$, the \emph{$p$-diversity-weighted portfolio} holds each firm in proportion to the $p$-th power of its capitalization: $p = 0$ is the equally weighted portfolio, $p = 1$ the market portfolio, and intermediate values interpolate between the two \cite[Section 3.4]{fernholz_stochastic_2002}. These are practical objects: massive amounts of money are held in funds that track capitalization-weighted and equal-weighted indices, and more specifically, a diversity-weighted portfolio with $p = 0.76$ was run commercially on the S\&P 500 \cite[Section 7.2]{fernholz_stochastic_2002}. We consider here the capitalization growth of these portfolios using the wave identity, adding consideration of turnover to the mean-field Atlas model of \cite[Section 5]{jourdain_capital_2015}. 
We follow the convention that entrants are bought at their entry weight, a self-financing rebalancing that does not move the portfolio's value. As we are tracking capitalization growth, dividends are not taken into account, and when a holding delists, the position is removed with no recovery (the differences for the total return case are collected in Remark~\ref{rem:total-return}). Throughout, a \emph{rule} is a weighting scheme, the map from a cross-section to portfolio weights, and a \emph{portfolio} is the position it produces. The observable we compute for each rule is its capitalization growth rate.

For the model without turnover, \cite{jourdain_capital_2015} show that the long-run growth rate of the wealth of the $p$-diversity rule is the integral of the arithmetic growth rate $b = \tilde b + \sigma^2/2$ (the drift coefficient of Subsection~\ref{subsec:finite-system}) against the \emph{$p$-weighted capital measure}, the probability measure on ranks that weights rank $s$ by the $p$-th power of the capitalization. Along our wave, that measure is given by
\begin{equation*}
    \bar\Pi^p(\d s) \;=\; \frac{e^{p\,\Xi(s)}}{\bar Z(p)}\,\d s,
    \qquad \bar Z(p) \;=\; \int_0^1 e^{p\,\Xi(r)}\,\d r,
\end{equation*}
with $\Xi = V^{-1}$ the quantile function of Remark~\ref{rem:quantile-ode}. 
Because $1 - V$ decays at the exponential rate $\alpha$ of Corollary~\ref{cor:tail}, the weight satisfies $e^{p\,\Xi(s)} = (1-s)^{-(p/\alpha)(1 + o(1))}$ as $s \to 1$, so $\bar Z(p) < \infty$ when $p < \alpha$ and $\bar Z(p) = \infty$ when $p > \alpha$. Hence the critical diversity index of \cite{jourdain_capital_2015} is the tail exponent of the wave, corrected for turnover. Beyond this index, the measure collapses onto the largest firms (the supercritical phase of \cite{chatterjee_phase_2010, jourdain_capital_2015}; see Remark~\ref{rem:diversity-loss} below). On the window 1975--2024, the measured inputs yield $\alpha = \AlphaRecWave$ (with bootstrap band $[\BootAlphaLo, \BootAlphaHi]$, see Subsection~\ref{subsec:reaction}) compared to \AlphaData{} in the data. Every rule with $0 \leq p \leq 1$ is therefore subcritical.

We now derive the rule's capitalization growth rate along the wave. In the particle system, the rule holds the weight $\pi_i = (X^i_t)^p / \sum_j (X^j_t)^p$ in firm $i$. Between events, It\^o's formula gives the drift of the portfolio's log value as $\sum_i \pi_i\, b(q_i) - \tfrac{1}{2} \sum_i \pi_i^2\, \sigma(q_i)^2$, as in the fixed universe \cite[Section 5]{jourdain_capital_2015}. An exit costs the portfolio its weight in the exiting firm: the log value jumps by $\log(1 - \pi_i)$, at rate $\lambda_d(q_i)$. Entries move the weights but not the value, by the convention above. Along the wave, the weighted empirical measure $\sum_i \pi_i\, \delta_{q_i}$ converges to $\bar\Pi^p$, and for $p < \alpha$ every individual weight vanishes by the finiteness of $\bar Z(p)$. Hence the quadratic It\^o correction and the gap between $\log(1 - \pi_i)$ and $-\pi_i$ vanish with the largest weight, leaving the growth rate
\begin{equation}\label{eq:gp-def}
    G^p \;=\; \langle b, \bar\Pi^p \rangle \;-\; \langle \lambda_d, \bar\Pi^p \rangle.
\end{equation}

\subsection{The Reduction Formula}

The proposition below reduces \eqref{eq:gp-def} to an explicit integral against the profile: substituting the wave identity for the drift and integrating the flux by parts leaves the front speed and the cross-sectional ingredients. This argument is not tied to power weights: for any rank-based weight density, the same integration by parts gives the rule's growth rate, with the derivative of the weight replacing the factor $p$ (compare the portfolio generating functions of \cite[Chapters 3 and 4]{fernholz_stochastic_2002} and \cite{karatzas_trading_2017}).

\begin{proposition}\label{prop:portfolio}
In the setting of Theorem~\ref{thm:wave}, let $\alpha$ be the tail exponent \eqref{eq:alpha-main} of Corollary~\ref{cor:tail}, and let $0 \leq p < \alpha$. Then
\begin{equation}\label{eq:reduction}
    G^p \;=\; c \;+\; (1 - p)\, G^p_* \;+\; \Big\langle \frac{\tilde f}{g},\, \bar\Pi^p \Big\rangle \;-\; \big\langle \lambda_d,\, \bar\Pi^p \big\rangle,
    \qquad
    G^p_* \;:=\; \tfrac{1}{2}\, \langle \sigma^2, \bar\Pi^p \rangle,
\end{equation}
where $G^p_*$ is the portfolio's excess growth rate. With $\tilde f = 0$ this is the reduction formula of \cite[Proposition 5.5]{jourdain_capital_2015}.
\end{proposition}

\begin{proof}
Since $b = \tilde b + \sigma^2/2$, it suffices to show $\langle \tilde b - c, \bar\Pi^p \rangle = -p\, G^p_* + \langle \tilde f/g, \bar\Pi^p \rangle$. Substitute the wave identity \eqref{eq:wave-identity-main}, $\tilde b - c = \varphi' + \tilde f/g$, and integrate the flux term by parts against $e^{p\Xi}$: using $\Xi' = 1/g$ and $\varphi/g = \sigma^2/2$,
\begin{equation*}
    \int_0^1 \varphi'\, e^{p\Xi}\,\d r
    \;=\; \big[\varphi\, e^{p\Xi}\big]_0^1 \;-\; p \int_0^1 \frac{\varphi}{g}\, e^{p\Xi}\,\d r
    \;=\; -\,p \int_0^1 \frac{\sigma^2}{2}\, e^{p\Xi}\,\d r.
\end{equation*}
The boundary terms vanish: at $s = 0$ the flux vanishes while $e^{p\Xi}$ stays bounded; at $s = 1$ the flux vanishes linearly by Theorem~\ref{thm:wave}(ii)--(iii), while $e^{p\,\Xi(s)} = (1-s)^{-(p/\alpha)(1 + o(1))}$ and $p < \alpha$ is strict. The ratio $\tilde f / g$ is bounded, with endpoint limits $\tilde f'(e)$ divided by the decay rate of $V'$ at $e$, again by Theorem~\ref{thm:wave}(ii), so every integral above converges. Dividing by $\bar Z(p)$ finishes the proof.
\end{proof}

Figure~\ref{fig:portfolio} applies the formula for $p\in[0,1]$ with inputs measured directly from the data. We use the rank-dependent volatility curve of Subsection~\ref{subsec:ranked-vol}, the fitted intensities of Section~\ref{sec:empirics}, and $\bar\Pi^p$ computed from the monthly cross-sections, all on the same window. The wave identity has eliminated $\tilde b$, so differences across $p$ are zero-calibration quantities in the sense of Subsection~\ref{subsec:ranked-vol}.

The figure gives an answer to the portfolio selection question of \cite[Section 5.6]{jourdain_capital_2015} on the half-century window: which $p$-diversity rule is optimal, given the volatility structure of the market? The stationary capital distribution of that theory does not exist for the measured coefficients (Subsection~\ref{subsec:beyond}), so the question could not previously be evaluated on data. Without turnover, a variance curve decreasing in rank makes $G^p$ decreasing in $p$, and the equally weighted portfolio is then optimal in capitalization growth among diversity-weighted rules \cite[Conclusion (C1)]{jourdain_capital_2015}. Turnover adds a force in the other direction: the exit intensity $\lambda_d$ falls by an order of magnitude from the bottom of the distribution to the top, so the exit drag reduces as $p$ grows. With the calibrated inputs, the volatility term dominates at every $p$. The curve is decreasing, and the equally weighted portfolio remains the optimal rule.

This verdict rests on the diversity of the calibrated market, which is a functional estimated by the model, rather than an assumption. By Corollary~\ref{cor:tail}, $\alpha > 1$ exactly when $b(1) - c < -\tilde f'(1)$. In words, the excess growth of the largest firms over the front must stay below the restoring rate of the reaction at the top. Recall that $\alpha = \AlphaRecWave$, so the market portfolio sits just inside the subcritical phase. The calibrated index is therefore consistent with the diversity hypothesis of stochastic portfolio theory, under which no single firm's market weight approaches one (\cite{fernholz_diversity_2005}; \cite{pal_geometry_2016} for the converse). However, diversity holds only barely; Remark~\ref{rem:diversity-loss} exhibits windows of failure.

\begin{figure}[t]
\centering
\includegraphics[width=0.82\textwidth]{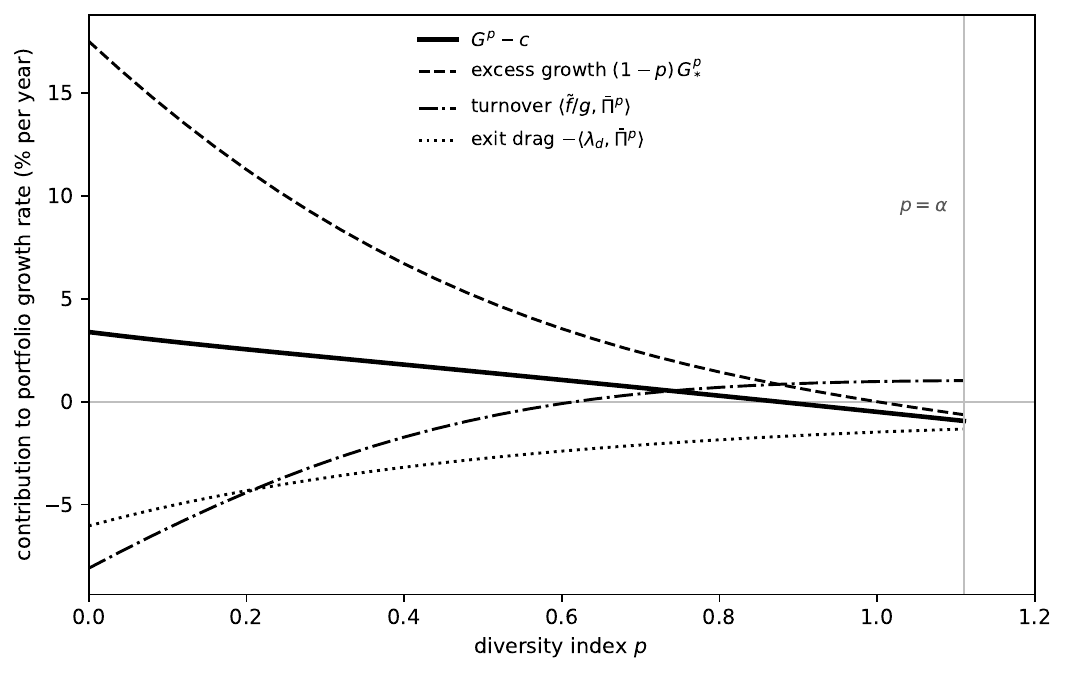}
\caption{Growth rates of $p$-diversity-weighted portfolios relative to the front speed: the reduction formula \eqref{eq:reduction} evaluated with the measured inputs, split into the excess-growth term, the turnover term, and the exit drag. The curves end at the critical diversity index $p = \alpha$ (vertical line), beyond which the mean-field formula fails (Remark~\ref{rem:diversity-loss}).}
\label{fig:portfolio}
\end{figure}

At $p = 0$, the turnover term is $B(1) - c$ by the speed identity \eqref{eq:wave-speed}, so $G^0 = \int_0^1 b(r)\,\d r - r_d$. In other words, the equally weighted portfolio gains the arithmetic growth rate averaged evenly across ranks, minus the aggregate exit rate. In numbers, the excess-growth term is $\tfrac{1}{2}\int_0^1 \sigma^2(r)\,\d r = \PortEWExcess\%$ per year; the average drift of firms relative to the front (Subsection~\ref{subsec:drift-calib}) is $\PortEWTurn\%$ on the trimmed rank grid of the measurement conventions (Subsection~\ref{subsec:ranked-vol}); and the exit rate is \PortEWDrag\%. The sum leaves the equally weighted portfolio \PortEWRel\% per year above the front.

At $p = 1$, the excess-growth term vanishes with the factor $1 - p$. The turnover terms simplify by an exact integration: substituting $s = V(\xi)$ and using $\tilde f' = \lambda_b - \lambda_d - f(1)$, the derivative
\begin{equation*}
    \frac{\d}{\d\xi}\big[\tilde f(V)\, e^{\xi}\big] \;=\; \big(\lambda_b(V) - \lambda_d(V) - f(1)\big)\, V'\, e^{\xi} \;+\; \tilde f(V)\, e^{\xi}
\end{equation*}
integrates to zero over $\R$, because $\tilde f(V)\, e^{\xi}$ vanishes at both ends. Notice that this is where $\alpha > 1$ enters. Hence $\langle \tilde f/g, \bar\Pi^1 \rangle + \langle \lambda_b - \lambda_d, \bar\Pi^1 \rangle = f(1)$, and \eqref{eq:reduction} at $p = 1$ becomes
\begin{equation}\label{eq:dilution}
    G^1 \;=\; c \;+\; f(1) \;-\; \langle \lambda_b, \bar\Pi^1 \rangle.
\end{equation}
Recall from the growth identity \eqref{eq:growth-identity} that total capitalization grows at $c + f(1)$ along the wave. Unlike the market portfolio, total capitalization gains the capital of entrants. Hence the market portfolio lags the growth of total capitalization by exactly the capital-weighted entry rate $\langle \lambda_b, \bar\Pi^1 \rangle$. In other words, the firm-count term of the growth identity is growth that a self-financing investor cannot capture. In numbers, the turnover term and the drag partially offset ($\PortMktTurn$ against $\PortMktDrag$), so the market portfolio grows near the speed of the front: $G^1 - c = \PortMktRel\%$ per year. Entry adds \PortEntryCountPct\% of firms per year but only \PortMktDilution\% of capital. The spread between the classical equal-weight and market portfolios is \PortEWMinusMkt\% per year. This is far below the \PortEWExcess\% suggested by the excess-growth term for $p=0$. Rebalancing toward the small-firm ranks gains the excess growth rate of the variance curve, but turnover takes most of it back through the sub-front drift of the overweighted ranks, and through their exit rates.

\begin{remark}[Loss of diversity]\label{rem:diversity-loss}
Without turnover, a diversity rule with $p$ above the critical index effectively holds the largest firm \cite[Remark 5.6]{jourdain_capital_2015}. In our model, the derivation of \eqref{eq:gp-def} fails there. For $p > \alpha$, the $p$-th powers of the capitalizations have tail exponent $\alpha/p < 1$, so in a finite sample drawn from the wave profile, the largest holding is a random positive fraction of the portfolio, and the law of that fraction depends on the wave only through $\alpha/p$ (the supercritical phase of \cite{chatterjee_phase_2010}); this is the sampling threshold discussed after Corollary~\ref{cor:tail}, applied for general $p$ rather than at the market weights. The portfolio is then no longer a mean-field object: the weighted empirical measure keeps an atom, and the exit charge $\log(1 - \pi_i)$ can no longer be linearized. The concentrated rule earns the growth rate of the largest firm while carrying that firm's exit risk. When we refit the intensities on the most recent fifteen years, 2010--2024, the reaction term stays bistable, but its top slope weakens to $\tilde f'(1) = \RecentFifteenFPrimeTopAnnual\%$ per year, so the half-life at the top lengthens from the fitted \RelaxHalfTopYears{} years to \RecentFifteenHalfTopYears{}. At the same volatility and drift, the tail formula returns $\alpha = \RecentFifteenAlpha$. Hence, on the coefficients of the concentration era itself, the market sits at the supercritical boundary. Windows reaching back into the 2000s lose the bistable structure altogether (the stronger failure of Subsection~\ref{subsec:beyond}).
\end{remark}

\begin{remark}[Caveats for the total return case]\label{rem:total-return}
Three caveats separate our analysis from an investor's total return. First, the coefficient $b$ is capitalization growth rather than total return: dividends sit outside the model, and share issuance enters the measured drift but not a shareholder's return. Second, the model takes the exit loss in full, while merger targets (\PctExitMerger\% of exits) are typically bought out at a premium and delistings for cause carry large negative returns \cite{shumway_delisting_1997, shumway_delisting_1999}. From the data, we see that an equally weighted exiting position recovers \PortRecEWCents{} cents on the dollar on average, and \PortRecDropCents{} cents for delistings for cause. Therefore, most of the assumed loss returns to the investor but not to the listed market. Third, we do not model transaction costs, and \eqref{eq:reduction} concerns the frictionless full-universe rules, whose rebalancing concentrates in the smallest and costliest ranks (see \cite{ruf_impact_2020}). The implementable top-set rules are treated in the next remark.
\end{remark}

\begin{remark}[Open markets and leakage]\label{rem:leakage}
Practical equal-weight rules restrict to a top set. For instance, we can consider a \emph{top-500 rule} (e.g.~the equal-weight S\&P 500), or a \emph{$85\%$-coverage rule} which holds the top set carrying $85\%$ of total capitalization (which is close in spirit to the MSCI USA Equal Weighted Index \cite{msci_inc_msci_2026}). For uniform weight on the ranks $[s, 1]$, the integration by parts in the proof of Proposition~\ref{prop:portfolio} leaves a boundary term,
\begin{equation*}
    G^0_{[s,1]} \;=\; c \;+\; \frac{1}{1-s} \lb( \int_s^1 \Big( \frac{\sigma^2(r)}{2} + \frac{\tilde f(r)}{g(r)} - \lambda_d(r) \Big) \,\d r \;-\; \varphi(s) \rb).
\end{equation*}
The new cost, $\varphi(s)/(1-s)$, is the diffusive flux through the membership boundary, using the formal mean-field reading of the collision local time at rank $s$ (as in Remark~\ref{rem:drift-estimation}). This is the \emph{leakage} of open-market portfolios identified in \cite[Example 4.3.5]{fernholz_stochastic_2002}. The membership trades are value-neutral, so the rule remains self-financing. The same integration by parts gives the growth $G^1_{[s,1]}$ of the capitalization-weighted $s$-coverage top set: its boundary term is the flux times the weight at the boundary, which is small for capitalization weighting by construction. In general, leakage is proportional to the weight a rule places at its own membership boundary, which is why the capitalization-weighted top set tracks its full-universe counterpart more closely than the equal-weight top set tracks its own.

Along the wave, a fixed level $s$ holds a fixed share of total capitalization, so the coverage rule is the mean-field object of the formula, while the top-500 rule's boundary moves whenever the number of firms changes.
We evaluate the top-500 rule and the $85\%$-coverage rule month by month at their own boundaries. The reduction formula's decomposition applies at a fixed level, so we can itemize the coverage rule's costs, but must report the top-500 rule as a single net number. Calibrated on 1975--2024, leakage is the dominant cost of top-set equal weighting: the coverage rule's excess growth of \PortTopCapExc\% per year is offset by \PortTopCapLeak\% of leakage and, with the smaller turnover terms, sits at \PortTopCapRel\% over the front. The top-500 rule also happens to sit at \PortTopRankRel\% over the front. We also note that the coverage rule's membership tracks the market's diversity: it held \PortTopCapCountStart{} firms on average in the late 1970s, peaked at \PortTopCapCountPeak{} in \PortTopCapPeakYear{}, and holds \PortTopCapCountLate{} in 2020--24.
\end{remark}

\subsection{The Market Portfolio over the Decades}

We now confront the simplified reduction formula for $p=1$ with the accounting of Subsection~\ref{subsec:growth-acct}, decade by decade. The bookkeeping of \eqref{eq:dilution} holds off the wave as well: the market portfolio grows as total capitalization does, minus the capital-weighted entry rate. Substituting the growth identity \eqref{eq:growth-identity} for the growth of total capitalization, the wave terms combine into the steady-state term, the deformation survives as the mean-over-median shape term, and CRSPs coverage growth enters as universe expansion. The decade accounting is therefore
\begin{equation}\label{eq:decade-ledger}
    \underbrace{\vphantom{\Big\langle}G^1 - c}_{\text{realized}}
    \;=\;
    \underbrace{\Big\langle \frac{\tilde f}{g} - \lambda_d,\, \bar\Pi^1 \Big\rangle}_{\text{steady-state term}}
    \;+\;
    \underbrace{\vphantom{\Big\langle}\frac{\d}{\d t}\log\frac{\text{mean cap}}{\text{median cap}}}_{\text{shape deformation}}
    \;+\; \text{universe expansion} \;+\; \text{residual},
\end{equation}
with every term in percent per year. Table~\ref{tab:port-acct} evaluates it on each decade. The realized row is $G^1 - c$ from those aggregates, with the front speed measured as the growth of the median capitalization. The steady-state term is the reduction formula \eqref{eq:reduction} at $p = 1$, computed with the decade's own intensities, quantile density, and capital weights. The shape term is measured as in Subsection~\ref{subsec:growth-acct}. Universe expansion is the growth of the raw CRSP firm count minus the cleaned net entry rate $f(1)$: the file's coverage expansions enter the realized side but are not economic entry. The residual is what remains: pure error, from the monthly discretization and from the estimated decade inputs.

\begin{table}[htbp]
\centering
\caption{The market portfolio against the front, decade by decade: the accounting \eqref{eq:decade-ledger}, in percent per year. Realized growth relative to the front is built from aggregates via \eqref{eq:dilution}; the steady-state term uses each decade's own inputs; the largest residual is \PortAcctResidMax.}
\label{tab:port-acct}
\setlength{\tabcolsep}{4pt}%
\begin{tabular}{lcccccccccc}
\toprule
\% per year & 1930s & 1940s & 1950s & 1960s & 1970s & 1980s & 1990s & 2000s & 2010s & 2020s \\
\midrule
realized $G^1 - c$ & $-2.8$ & $-2.4$ & $+1.5$ & $+7.5$ & $+10.4$ & $+7.6$ & $+0.3$ & $-12.1$ & $-0.5$ & $+14.0$ \\
steady-state term & $-0.1$ & $+1.0$ & $+0.4$ & $+1.2$ & $-0.7$ & $+1.7$ & $+0.9$ & $-4.8$ & $-1.2$ & $+1.0$ \\
shape deformation & $-3.2$ & $-4.3$ & $+1.5$ & $+0.9$ & $+2.9$ & $+5.9$ & $-0.6$ & $-6.4$ & $+1.5$ & $+13.3$ \\
universe expansion & $+0.4$ & $+0.9$ & $-0.6$ & $+5.1$ & $+8.2$ & $-0.1$ & $+0.0$ & $-1.0$ & $-0.9$ & $-0.4$ \\
residual & $+0.0$ & $+0.0$ & $+0.1$ & $+0.2$ & $-0.0$ & $+0.0$ & $-0.0$ & $+0.1$ & $+0.0$ & $+0.0$ \\
\bottomrule
\end{tabular}

\end{table}

We read Table~\ref{tab:port-acct} in two regimes. In the on-wave decades, the shape row is small and the steady-state term accounts for the realized excess. In the concentration decades the shape row carries the gap, and in the monostable 2000s the steady-state term captures the sign and about half the size of the realized shortfall. The residual row is the falsifiable part: no row is fitted to another, so the residual has no reason to be small unless the decade estimates are mutually consistent and the tail satisfies $\alpha > 1$. It stays within \PortAcctResidMax\% per year in every decade. We run this construction at $p = 1$ because the drift cancels from both sides, and because the market portfolio's value is a functional of the cross-section, so no rebalancing must be tracked. At $p = 0$ the prediction contains the average drift, which the data cannot supply independently.

\section{Discussion and Conclusion}
\label{sec:discussion}

\paragraph{Mergers as Transport.}
Our model treats a merger as a death, which correctly removes the firm but ignores where its capital goes. A more faithful model would add a merger kernel $\kappa(q, q')$ for absorption of rank $q$ by rank $q'$, with the acquirer jumping to the combined size. In the mean-field limit this adds a nonlocal transport term to \eqref{eq:pde}; finite systems of competing Brownian particles with splits and mergers of exactly this kind are constructed and analyzed in \cite{karatzas_diverse_2016}. Since mergers are \PctExitMerger\% of exits, this is the natural next refinement; the merger check of Subsection~\ref{subsec:fitting} shows that the effect on the death intensity's shape is modest, but the upward jumps of acquirers are absent from the present model.

\paragraph{Regime Switching.}
Figure~\ref{fig:f-decades} shows that the intensities factor, to a good approximation, into stable rank shapes modulated by time-varying levels. Letting these levels follow a hidden state would yield a regime-switching model. Such a model would seek to explain the deformation of the wave exhibited in Subsection~\ref{subsec:growth-acct}.
For a single permanent change in the coefficients, Theorem~\ref{thm:relaxation} and Remark~\ref{rem:rank-dependent} tell us that the market converges from the current cross-section to the new wave profile on the timescales of the new reaction term. For slowly varying intensities, the solution should track the wave of the instantaneous coefficients up to a lag set by the timescales of Remark~\ref{rem:clock}, largest at the leading edge, and the episode analysis of Subsection~\ref{subsec:growth-acct} would be accounted for via the tracking error. Furthermore, a regime model would imply dynamics for the diversity boundary from Section~\ref{sec:diversity}, with implications for portfolio rules, especially when the tail index sits at the supercritical boundary (Remark~\ref{rem:diversity-loss}).

\paragraph{Common Noise and Fluctuations.}
Theorem~\ref{thm:mean-field} is a law of large numbers, and the natural next step is to consider fluctuations: $\sqrt{n}$ times the difference between the empirical distribution and its limit. For rank-based diffusions without entry and exit this program is carried out in \cite{kolli_spde_2018}, where the global fluctuations converge to a Gaussian stochastic PDE, and the corresponding large deviations are established in \cite{dembo_large_2016}; with common noise the law-of-large-numbers limit is itself stochastic \cite{kolli_large_2019}, treated most recently in the pathwise entropy framework of \cite{shkolnikov_rank-based_2026}. Adding branching to these limits is a subject for future inquiry. 

\paragraph{Conclusion.}
We proposed a particle-system model of the capital distribution, with geometric Brownian firms entering and exiting at rank-dependent rates, and we proved that its empirical distribution converges, in the Kolmogorov distance and uniformly on compact time intervals, to the solution of an explicit reaction-diffusion equation. Calibrated on a half-century of CRSP data, the reaction term is bistable, and the long-run capital distribution is a traveling wave. With volatility measured rank by rank and no calibrated drift parameters, the wave tracks the empirical capital distribution curve in every decade of the sample, within tens of percent at the top ranks, and the growth accounting shows the market growing by translation and entry outside of concentration episodes. The wave identity recovers the drift, which time series cannot usefully identify. The endpoint slopes of the reaction term enter the tail exponent and set the relaxation timescales, and the persistence of shape deviations across the century matches these timescales in ordering and in scale. The recovered drift violates the chord condition of the drift-stabilized models at every rank; the calibrated market is turnover-stabilized. Evaluating diversity weighted portfolios along the wave, we see that the equally weighted portfolio gains the excess growth rate of the variance curve, but turnover takes most of it back. Market diversity is an output of our calibration rather than an assumption: the half-century market sits just inside the diverse phase, and the coefficients of the most recent fifteen years sit at its boundary.

\section*{Acknowledgments}

This work was supported by the Columbia College Class of 1939 Research Fellowship (Summer 2025). We thank Steven Campbell and David Itkin for feedback on an early draft, and Philipp Jettkant for insightful discussion.

\appendix

\section{Validating the Mean-Field Limit}
\label{sec:validation}

Theorem~\ref{thm:mean-field} is an asymptotic statement, and we verify that it describes systems of realistic size. Throughout this section we work in the constant-coefficient case of Corollary~\ref{prop:constant}, whose two scalar parameters are the most transparent to calibrate. We simulate the particle system of Subsection~\ref{subsec:finite-system} with the calibrated intensities and compare its empirical CDF with a finite-difference solution of the limiting PDE.

\paragraph{Setup.} We work in log coordinates and in months. The volatility and growth parameters are the scalar coefficients of the measurement conventions in Subsection~\ref{subsec:ranked-vol}: the cross-sectional median over firms of the mean monthly change in log capitalization, $\tilde\mu = \MuTildeMonthly$ per month, and of the realized volatility of monthly log returns, $\sigma = \SigmaMonthly$ per square-root month ($\SigmaAnnual$ annualized), both on 1975--2024. The initial condition is the empirical log-capitalization distribution in \InitMonth. The particle system starts from $n = \NMain$ particles and uses an Euler scheme with step $\Delta t = 0.05$ months. At each step, every particle moves by $\tilde\mu \Delta t + \sigma\sqrt{\Delta t}\, Z$, ranks are recomputed within the current population, and each particle gives birth with probability $\lambda_b(q_i)\Delta t$ and dies with probability $\lambda_d(q_i)\Delta t$. The PDE \eqref{eq:pde} is solved in log coordinates by a Crank--Nicolson scheme with the reaction term treated explicitly, on a grid padded so that the mass reaching the boundaries is negligible relative to the reported errors; the padding is six standard deviations of the accumulated diffusion plus the drift displacement, and the boundary values are held at $0$ and $1$.

\paragraph{Results.} At horizons of 5, 10, and 20 years, the normalized empirical CDF of the particle system and the PDE solution agree to within sup-norm distances of \SupErrSixty, \SupErrOneTwenty, and \SupErrTwoForty{}, respectively: fluctuations of the size the theorem predicts for $n = \NMain$, and too small to be visible in an overlay, so we report the numbers rather than the curves. The particle count tracks the deterministic mass curve $m(t) = e^{f(1)t}$ (Figure~\ref{fig:count-error}, left), ending at $\GrowthSim$ times its initial size against a predicted $\GrowthTheory$. Repeating the simulation across initial population sizes from \NSweepSmall{} to \NSweepBig{} (Figure~\ref{fig:count-error}, right), the sup-norm error at the 10-year horizon (each point the mean over ten simulations) falls from \SweepErrSmall{} to \SweepErrBig, consistent with the $n^{-1/2}$ rate expected from the martingale bounds in the proof.

\begin{figure}[htbp]
\centering
\begin{minipage}{0.49\textwidth}
\centering
\includegraphics[width=\textwidth]{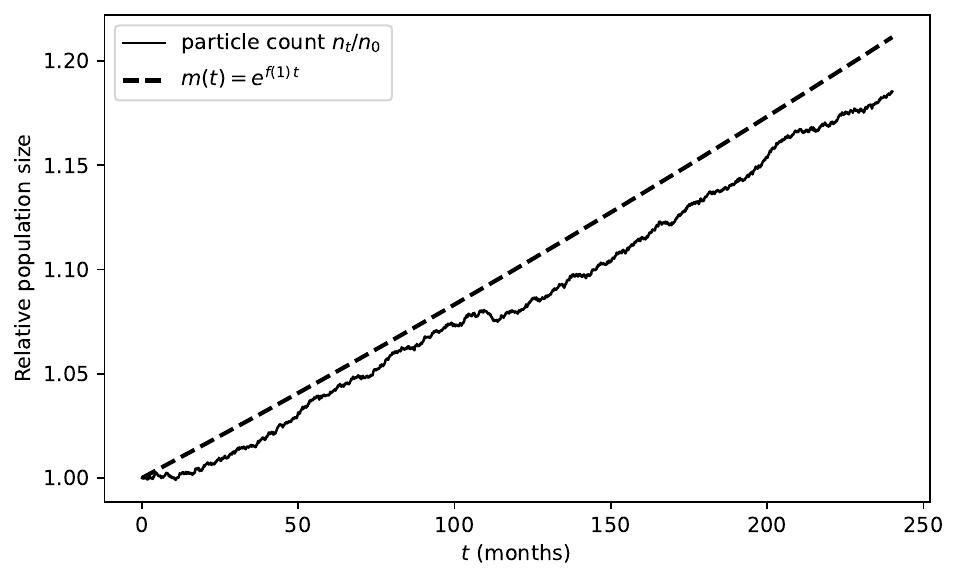}
\end{minipage}
\hfill
\begin{minipage}{0.49\textwidth}
\centering
\includegraphics[width=\textwidth]{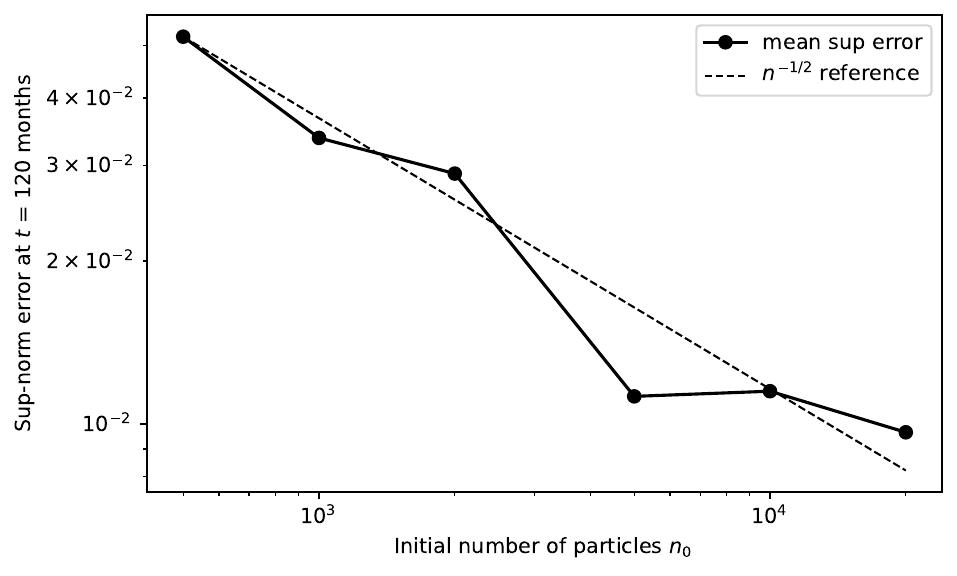}
\end{minipage}
\caption{Left: relative population size in the simulation against the deterministic limit $e^{f(1)t}$. Right: sup-norm distance between the particle system and the PDE at $t = 120$ months, as the initial population size varies; the dashed line is the $n^{-1/2}$ reference slope.}
\label{fig:count-error}
\end{figure}

\section{Proofs of the PDE Lemmata}
\label{app:pde-proofs}

\begin{proof}[Proof of Lemma~\ref{lem:uniqueness}]
Let $w_1, w_2$ be admissible solutions with the same initial CDF, write $a = \sigma^2/2$, and set $\Delta = w_1 - w_2$. Then $|\Delta| \leq 1$, $\Delta(0, \cdot) = 0$, and $\|\Delta(t, \cdot)\|_{L^1(\R)} = W_1(\pi^1_t, \pi^2_t)$, the Wasserstein distance between $\pi^1_t$ and $\pi^2_t$, is finite with $\int_0^T \|\Delta(t)\|_{L^1}\, \d t < \infty$, by the moment condition. Hadamard's formula $A(w_1) - A(w_2) = \bar a\, \Delta$ with $\bar a(t,y) = \int_0^1 a\big(\theta w_1 + (1-\theta) w_2\big)(t,y)\, \d\theta$, and likewise for $B$ and $\tilde f$, shows that $\Delta$ solves the \emph{linear} equation
\begin{equation*}
    \pp_t \Delta \;=\; \pp^2_{yy}\big(\bar a\, \Delta\big) - \pp_y \big(\bar b\, \Delta\big) + \bar c\, \Delta
\end{equation*}
in the sense of distributions, where $\bar a \in [\min a, \max a]$, $|\bar b| \leq \sup|\tilde b|$, and $|\bar c| \leq \mathrm{Lip}(\tilde f)$ are measurable. Testing against $g$ and using $\Delta(0, \cdot) = 0$, this gives
\begin{equation}\label{eq:dual-pairing}
    \int_\R \Delta(t)\, g(t)\, \d y \;=\; \int_0^t \!\!\int_\R \Delta\, \big[ \pp_s g + \bar a\, \pp^2_{yy} g + \bar b\, \pp_y g + \bar c\, g \big]\, \d y\, \d s
\end{equation}
for $g \in C_c^\infty([0,T] \times \R)$, and by mollification for every $g \in C^{1,2}$ that is bounded with two bounded derivatives and decays as $|y| \to \infty$ (compare \cite[Remark A.1]{jourdain_propagation_2013}). The identity holds at every $t$, not only almost every $t$: writing $\int_\R \Delta(t)\, g\, \d y = \int_\R \big( \int_z^\infty g(t, y)\, \d y \big)\, (\pi^1_t - \pi^2_t)(\d z)$, both sides of \eqref{eq:dual-pairing} are continuous in $t$ by the weak continuity of $t \mapsto \pi^i_t$.

Fix $t < T$ and $\zeta \in C^\infty_c((0,t) \times \R)$. Mollify the coefficients in $(s,y)$: $\bar a_\varepsilon = \bar a * \xi_\varepsilon$, and likewise $\bar b_\varepsilon, \bar c_\varepsilon$; the mollifications obey the same bounds, are smooth, and converge almost everywhere. The terminal-value problem
\begin{equation*}
    \pp_s g + \bar a_\varepsilon\, \pp^2_{yy} g + \bar b_\varepsilon\, \pp_y g + \bar c_\varepsilon\, g \;=\; \zeta, \qquad g(t, \cdot) = 0,
\end{equation*}
has a unique bounded classical solution $g_\varepsilon$ \cite[Chapter IV, \S5, Theorem 5.1]{ladyzenskaja_linear_1968} (after the time reversal $s \mapsto t - s$ it is a Cauchy problem with smooth bounded coefficients; we extend the coefficients past the endpoints of the time interval before mollifying, and uniqueness among bounded solutions follows from the maximum principle). Two bounds hold uniformly in $\varepsilon$. First, the Feynman--Kac representation \cite[Theorem 5.7.6]{karatzas_brownian_1991} of $g_\varepsilon$ along the nondegenerate diffusion with coefficients $\bar b_\varepsilon$ and $\sqrt{2 \bar a_\varepsilon}$ gives $\|g_\varepsilon\|_\infty \leq T\, \|\zeta\|_\infty\, e^{\mathrm{Lip}(\tilde f)\, T}$ (the theorem asks for a nonnegative potential; tilting by $e^{-\mathrm{Lip}(\tilde f)\, s}$ arranges it and produces the exponential factor), together with Gaussian decay in $y$; the reaction term of the original equation enters the proof only through this exponential factor. Second, multiplying the equation by $\pp^2_{yy} g_\varepsilon$ and integrating over $(s,t) \times \R$: the first term contributes $\tfrac{1}{2}\|\pp_y g_\varepsilon(s)\|^2_{L^2}$ with a favourable sign, the terminal condition removes the boundary term, and $\bar a_\varepsilon \geq \min a > 0$ with Young's and Gronwall's inequalities gives
\begin{equation*}
    \sup_{s \leq t}\, \|\pp_y g_\varepsilon(s)\|_{L^2(\R)} + \|\pp^2_{yy} g_\varepsilon\|_{L^2((0,t) \times \R)} \;\leq\; C,
\end{equation*}
with $C$ depending on $\min a$, the coefficient bounds, $T$, and $\zeta$, but not on $\varepsilon$. The integrations by parts here are justified because $g_\varepsilon(s)$ and $\pp_y g_\varepsilon(s)$ are square-integrable with $\varepsilon$-uniform bounds; this is what the Gaussian decay from the representation provides.
Now insert $g_\varepsilon$ into \eqref{eq:dual-pairing}: since $g_\varepsilon(t, \cdot) = 0$ and $g_\varepsilon$ solves its equation,
\begin{equation*}
    \int_0^t \!\!\int_\R \Delta\, \zeta \;=\; -\int_0^t \!\!\int_\R \Delta\, \big[ (\bar a - \bar a_\varepsilon)\, \pp^2_{yy} g_\varepsilon + (\bar b - \bar b_\varepsilon)\, \pp_y g_\varepsilon + (\bar c - \bar c_\varepsilon)\, g_\varepsilon \big].
\end{equation*}
Each term vanishes as $\varepsilon \to 0$. For the first, Cauchy--Schwarz bounds it by $\|\Delta (\bar a - \bar a_\varepsilon)\|_{L^2} \|\pp^2_{yy} g_\varepsilon\|_{L^2}$; the second factor is bounded uniformly, and since $|\Delta (\bar a - \bar a_\varepsilon)|^2 \leq (2 \max a)^2\, |\Delta| \in L^1((0,t) \times \R)$ while $\bar a_\varepsilon \to \bar a$ almost everywhere, the first factor tends to $0$ by dominated convergence. The other two terms are handled the same way, with the uniform bounds on $\pp_y g_\varepsilon$ and $g_\varepsilon$. Hence $\int_0^t\!\!\int \Delta \zeta = 0$ for every such $\zeta$ and every $t < T$, so $\Delta = 0$ almost everywhere; since each $w_i(t, \cdot)$ is right-continuous in $y$ and weakly continuous in $t$, $\Delta \equiv 0$.
\end{proof}

\begin{proof}[Proof of Lemma~\ref{lem:regularity}]
Consider \eqref{eq:pde-general} in divergence form:
\begin{equation*}
    \pp_t \bar w \;=\; \pp_y\Big(\frac{\sigma^2(\bar w)}{2}\,\pp_y \bar w\Big) - \pp_y B(\bar w) + \tilde f(\bar w), \qquad \bar w(0, y) = w_0(y).
\end{equation*}
The equation is uniformly parabolic with $C^1$ coefficients, and the initial data is bounded and continuous, so the Cauchy problem has a bounded strong solution $\bar w$ for $t > 0$, with $\bar w$ and $\pp_y \bar w$ continuous \cite[Chapter V, \S6, Theorem 6.3]{ladyzenskaja_linear_1968} (we solve the first boundary value problems on $[-R, R]$ with boundary data $0$ and $1$, replacing $w_0$ by a mollified monotone version equal to $0$ near $-R$ and $1$ near $R$: the corner compatibility conditions then hold because $\tilde f(0) = \tilde f(1) = 0$, and the modification is confined to $\{|y| \geq R - 1\}$, so for each compact set the data is eventually unmodified and the limit attains the true data $w_0$; the interior estimates of \cite[Chapter V, \S6]{ladyzenskaja_linear_1968} depend only on the distance to the parabolic boundary, hence are uniform in $R$, and a diagonal extraction gives the solution on $\R$). We check that $\bar w$ is admissible. Comparison with the constant solutions keeps $\bar w$ in $[0, 1]$. Monotonicity in $y$ is preserved by comparison: for $h > 0$ the translate $\bar w(t, y + h)$ solves the same equation with initial datum $w_0(y+h) \geq w_0(y)$, hence stays above $\bar w$. The first moment propagates as in Step 2 of the proof of Theorem~\ref{thm:mean-field}, with the same truncation argument run on the equation instead of the particle system: testing against smooth truncations of $|y|$ and applying Gronwall's inequality bounds the first moment on $[0,T]$, and in particular no mass reaches infinity, so the monotone limits of $\bar w(t, \cdot)$ at $\mp\infty$ remain $0$ and $1$. Weak continuity in $t$ holds because the pairing with a test function is the time integral of bounded terms. Finally, since $\pp_y \bar w$ is continuous, the chain rule turns the divergence form into the double-divergence form \eqref{eq:pde-general} in the sense of distributions. Hence $\bar w$ is admissible, which proves existence.

For the regularity, each $\bar w(t, \cdot)$ with $t > 0$ is continuous, nondecreasing, and has limits $0$ and $1$: a continuous CDF; at $t = 0$ this is the continuity of $w_0$. For the continuity in time, the marginals are weakly continuous and every $\bar w(t, \cdot)$ is a continuous CDF, so Lemma~\ref{lem:polya} applied along $t_k \to t$ upgrades weak convergence of the marginals to sup-norm convergence. Sup-norm continuity in time and continuity in $y$ at each time yields joint continuity.
\end{proof}

\printbibliography[heading=bibintoc]

\end{document}